\documentclass{scrartcl}

\usepackage{iftex}
\usepackage{comment}

\ifPDFTeX 
  \usepackage[utf8]{inputenc}[0,c]
  \usepackage[T1]{fontenc}
  \usepackage{lmodern}
\else
  \usepackage[no-math]{fontspec}
  
\fi

\usepackage[british]{babel}
\usepackage{etoolbox,xifthen,xparse}

\usepackage[dvipsnames]{xcolor}
\usepackage{aliascnt,enumitem}

\usepackage{amsmath}
\usepackage{amsfonts,amsthm,mathtools,booktabs}
\ifPDFTeX
  \usepackage{bbm}
  \usepackage{mathrsfs}
\else
  \usepackage{unicode-math}
  \let\mathbbm\mathbb

\fi
\usepackage[retainorgcmds]{IEEEtrantools}

\usepackage{url,doi}
\usepackage[numbers]{natbib}

\usepackage{hyperref,bookmark}

\hypersetup{hidelinks}

\newcommand{\newsstheorem}[2]{
  \newaliascnt{#1}{dummy}
  \newtheorem{#1}[#1]{#2}
  \aliascntresetthe{#1}
  \expandafter\def\csname #1autorefname\endcsname{#2}
}
\theoremstyle{plain}
\newsstheorem{theorem}{Theorem}
\newsstheorem{proposition}{Proposition}
\newsstheorem{corollary}{Corollary}
\newsstheorem{lemma}{Lemma}
\theoremstyle{definition}
\newsstheorem{definition}{Definition}
\newsstheorem{example}{Example}
\newsstheorem{assumption}{Assumption}
\theoremstyle{remark}
\newsstheorem{remark}{Remark}

\setlist[enumerate,1]{label={(\roman*)}}
\setlist[enumerate,2]{label={(\alph*)}}
\setlist[enumerate,3]{label={(\Roman*)}}

\makeatletter \newcommand\mathof[1]{{\operator@font#1}} \makeatother
\newcommand\dd{\mathof{d}}

\newcommand\defbb[1]{\expandafter\newcommand\csname #1#1\endcsname{\mathbb{#1}}}
\defbb{N} \defbb{Z} \defbb{Q} \defbb{R} \defbb{C}
\defbb{P} \defbb{E}

\newcommand\Dd{\mathcal{D}}
\newcommand\Uu{\mathcal{U}}
\newcommand\Aa{\mathcal{A}}

\newcommand\FF{\mathscr{F}}

\newcommand\Zb{\mathbf{Z}}

\newcommand\kr{\mathtt{k}}
\newcommand\rr{\mathsf{r}}

\newcommand\Ind{\mathbbm{1}}
\newcommand\Indic[1]{\Ind_{\{#1\}}}
\newcommand\from{\colon}

\DeclareMathOperator{\arginf}{arg\,inf}

\DeclarePairedDelimiter\norm{\lVert}{\rVert}
\newcommand\mass[1]{\norm{#1}}

\providecommand{\email}[1]{\href{mailto:#1}{\nolinkurl{#1}}}
\providecommand{\arxivref}[2]{\href{http://arxiv.org/abs/#1}{arXiv:#1}%
  \ifblank{#2}{}{ [#2]}%
}

\usepackage[colorinlistoftodos]{todonotes}

\presetkeys{todonotes}{size=\scriptsize}{}

\title{Strong laws of large numbers for a growth-fragmentation process with bounded cell sizes}
\author{Emma Horton%
  \footnote{INRIA, Bordeaux Research Centre, Talence, France. \email{emma.horton@inria.fr}}%
  \and Alexander R.\ Watson%
  \footnote{Department of Statistical Science, University College London, UK. \email{alexander.watson@ucl.ac.uk}}
}
\date{\today}

\begin{document}
\maketitle
\begin{abstract}
  Growth-fragmentation processes model systems of cells that grow continuously over time 
  and then fragment into smaller pieces. 
  Typically, on average, the number of cells in the system
  exhibits asynchronous exponential growth and,
  upon compensating for this, 
  the distribution of cell sizes 
  converges to an asymptotic profile.
  However, the long-term 
  \emph{stochastic} behaviour of the system is more delicate, and
  its almost sure asymptotics have been so far largely unexplored.
  In this article, we study a growth-fragmentation process 
  whose cell sizes are bounded above, 
  and prove the existence of regimes with differing
  almost sure long-term
  behaviour.
\end{abstract}
{\small

}

\section{Introduction}
\label{s:intro}

Growth-fragmentation refers to a collection of mathematical models in which
objects -- 
classically, biological cells --
slowly gather mass over time, and fragment suddenly into multiple, smaller offspring.
In recent years, there has been a growing interest in probabilistic models, in the form of
growth-fragmentation processes. These have been developed in the framework of
piecewise-deterministic Markov processes by several authors \cite{AB18,CDGMMY,BW-gf-fk},
and in a very general form by Bertoin \cite{BeMGF}.

In this article, we study a particularly simple growth-fragmentation process where the
long-term behaviour can be described explicitly. In this process, the cells grow
exponentially (that is, the growth rate is linear) until they reach a certain mass,
at which point growth stops abruptly. Fragmentation occurs at a constant rate regardless
of the cell mass, and the relative masses of offspring are similarly independent of the parent mass.
We are interested in understanding the long-term behaviour of these systems, and in
particular whether, with appropriate rescaling, they can converge to some
`asymptotic profile', describing the masses of cells in the limit.

This process can be regarded as a basic model for cells with a fixed bound on their mass,
and this property is attractive in terms of applications
\cite{CCF-links,BPR12},
either to account in an indirect
way for resource competition, or to model a biologically determined bound on how large
cells seek to become.
Alternatively, from a mathematical perspective, we can see this process as being part
of a family of perturbations of a fundamental \emph{homogeneous} growth-fragmentation process,
in which all rates are independent of mass. This homogeneous case was studied in
\cite{BW-gfe}, where it was observed that no such simple asymptotic profile exists.
However, Cavalli \cite{cavalli2019} has shown that a simple change in the drift,
making it piecewise constant, can yield convergence of averages
to an explicit asymptotic profile.
In this work, we use this approach to take the results a step further: not only do we
obtain explicit expressions and show that the system averages exhibit this convergence,
but we are also able to prove this behaviour for the stochastic system of cell masses via
a strong law of large numbers.

Work in this field often begins with the \emph{growth-fragmentation equation},
which in general form is given by
\[ \partial_t u_t(x) + \partial_x \bigl( \tau(x) u_t(x) \bigr) + (B(x) + D(x)) u_t(x)
  = 
  \int_{(x,\infty)} k(y,x) u_t(y)\, \dd y,
\]
with some initial condition $u_0 = g_0$.
The equation can be usefully rewritten in its weak form, in which we have a 
collection of measures $(\mu_t)_{t\ge 0}$, started from $\mu_0 = \delta_x$,
satisfying
\begin{equation}
  \label{e:gfe-weak}
  \partial_t \langle f, \mu_t \rangle = \langle \Aa f, \mu_t \rangle,
  \quad
  \Aa f(x) = \tau(x) f'(x) + \int_{(0,x)} f(y) \, k(x,\dd y) - \bigl(B(x) + D(x)\bigr)f(x),
\end{equation}
where $\langle f,\mu\rangle = \int f \dd \mu$,
$\tau$ is the growth rate, $B(x) = \int_{(0,y)} \frac{y}{x} k(x,\dd y)$ is the fragmentation rate,
$D$ is the killing rate 
and $k$ is a kernel describing the masses of offspring cells
given the mass of the parent.

The specific case we will consider is when
$\tau(x) = ax\Indic{x< c}$,
meaning that the constant $c>0$ is the maximum possible cell mass,
with cells growing exponentially until reaching mass $c$.
We take the other coefficients to be independent of cell mass:
$B(x) = B$ and $D(x) = \mathtt{k}$ are constant,
and we express the offspring mass distribution in terms of a measure $\rho$
on $(0,1)$,
such that $\int f(y) \, k(x,\dd y) = \int f(xy) \, \rho(\dd y)$.
The measure $\rho$ describes the \emph{relative} masses of offspring cells,
which are independent of the mass of the parent,
and it satisfies $\int_{(0,1)} y \, \rho(\dd y) = B$.
The operator $\Aa$ is then given by
\begin{equation}
  \Aa f(x) = axf'(x) + \int_{(0,1)} f(xy) \, \rho(\dd y) - (B+\mathtt{k}) f(x),
  \label{GFeq}
\end{equation}
where the domain $\Dd(\Aa)$ satisfies
\[ 
  \Dd(\Aa) \supset \{ f \in C^1_c(0,c] : f'(c) = 0 \}.
\]

We will shortly see that these models represent an averaged version of
the stochastic system of cell masses which forms our main object of study.
To describe this system, let us start by introducing a stochastic process $X$
started at $x$, which can be
thought of as the trajectory of a single cell with initial mass $x$.
Write $\PP_x$ for the probability measure associated with this initial mass.
First, let $\xi$ be the process
\[ \xi_t = y + at + \sum_{i = 1}^{\mathcal{N}_t} J_i, \qquad t < \zeta, \]
where $y = \log x$,
$a > 0$ represents the (exponential) growth rate, $\mathcal{N}$ is a Poisson process with rate $B$,
and $(J_i)_{i\ge 1}$ are independent, identically distributed jump sizes whose
support is contained in $[-\log 2, 0)$.
Let $\zeta$ be an exponential random variable with rate $\mathtt{k}$, and let $\xi$ be sent to some cemetery
state at time $\zeta$.
Define $\Pi(\dd x) = B\PP(J_1 \in \dd x)$. The process $\xi$ is a Lévy process
with drift $a$, no Gaussian part, Lévy measure $\Pi$ and killing rate $\kr$.
Let $\bar{\xi}_t = \sup\{ \xi_s : s \le t \}$, take $b = \log c$ and define
\[ \xi^b_t = \xi_t - (\bar{\xi}_t - b) \vee 0 , \]
which is the process $\xi$ reflected above at $b$.

Our single-cell trajectory is given by $X_t = \exp(\xi_t^b)$, the exponential
of a reflected Lévy process, starting at $X_0 = x$. This trajectory
increases with exponential rate $a$ until reaching $c=e^{b}$,
has negative jumps and is killed at rate $\mathtt{k}$.
We understand that, if $\xi_t$ is in the cemetery state, then the same applies
to $X_t$.
Each jump of size $\Delta X_t < 0$ represents a splitting event, and is
associated with the birth of another cell (i.e. a binary fission event),
with initial mass
$-\Delta X_t$ and where the trajectory of its mass is given by a copy of $X$
which is conditionally independent given its starting value.
By introducing these new cells at every jump of $X$,
and iterating this process with each successive set of children,
we build up the growth-fragmentation process.
In terms of Bertoin's Markovian growth-fragmentation processes
\cite{BeMGF}
this is a growth-fragmentation with `cell process' $X$.

In order to describe the masses of each cell, we define
$Z_u(t)$ to be the mass of the cell with label $u$ at time $t$;
the set of labels $u$ will be described later.
The \emph{growth-fragmentation process},
$\Zb(t) = \sum_{u} \delta_{Z_u(t)}$,
is a point measure-valued process describing the masses of all cells
alive at each time-point.

Crucially, the relation between $\Zb = (\Zb(t))_{t\ge 0}$ and $(\mu_t)_{t\ge 0}$
is given by
\begin{equation}
  \langle f, \mu_t \rangle = \EE_x[ \langle f, \Zb(t) \rangle ] ,
  \label{e:mu-Psi}
\end{equation}
for bounded measurable $f$,
with $\rho = \Pi\circ g^{-1} + \Pi\circ\bar{g}^{-1}$, where $g(x) = e^x$
and $\bar{g}(x) = 1-e^{x}$.
In other words, as indicated earlier, $\mu_t$ represents the average
linear behaviour of the growth-fragmentation process.

There is a vast literature on the long-term behaviour
of growth-fragmentation
equations and processes (see 
\cite{BW-gfe,MS-gfe,DbiecEtAl,Bernard_2019}
and \cite{CCF-gf,Dadoun:agf,BeMGF},
respectively, 
to name just a few references.)
From the analytical perspective, the conventional approach is to find $\lambda \in \RR$, 
a positive function $h$ and a measure $\nu$, such that
\begin{equation}
  \nu \mathcal{A} = \lambda \nu \qquad \text{ and } \qquad \mathcal{A}h = \lambda h,
  \label{eigen}
\end{equation}
where $\langle \nu, h \rangle = 1$, using the
Krein-Rutman theorem 
\cite{MS-gfe, BCGM, Per06}
or the Lumer-Phillips theorem \cite{Bernard_2019}.
Generalised relative entropy methods \cite{DbiecEtAl, MMP-gre}
can be used to show the long-term asynchronous exponential growth
\begin{equation}
  \langle f,\mu_t \rangle \sim e^{\lambda t} \langle f,\nu\rangle h(x).
  \label{e:long-term-mu}
\end{equation}
On the other hand, from a probabilistic perspective, 
it is often more convenient to work directly 
with the expectations of the growth-fragmentation
process $\Zb$ to show that the process satisfies
\begin{equation}
  \EE_x[\langle f, \Zb(t)\rangle] \sim e^{\lambda t}\langle \nu, f\rangle h(x),
  \label{PF}
\end{equation}
which, thanks to \eqref{e:mu-Psi}, agrees with \eqref{e:long-term-mu}.
One approach used in \cite{BW-gfe, cavalli2019} is to 
consider a single tagged particle via a Feynman--Kac formula, 
and to analyse the Laplace transform of the hitting time of points of this particle.
Other recent approaches make use of Harris-style theorems for 
nonconservative semigroups \cite{BCGM}, and quasi-stationary methods \cite{CV16}.

Finally, another approach, \cite{BPR12,CGY20}, which can be seen as a hybrid 
of the functional analytic and stochastic perspectives,
involves identifying $h$ through some
variation of the analytic approaches above, using it to derive a stochastic
semigroup, and then checking the ergodic theorem to obtain
a probability measure, which is a simple transformation of the eigenmeasure
$\nu$, leading to the desired long-term behaviour.
This idea is the starting point for our work.

The main aim of this article is to describe the long-term \textit{stochastic} behaviour of the 
growth-fragmentation process. To this end, we will prove
two theorems that can be seen as stochastic analogues of the deterministic asymptotic
\eqref{PF}, and which, indeed, imply it.
The first of these theorems is the following \emph{strong law of large numbers}.

\begin{theorem}\label{t:slln}
  Assume that
  \[
    B > \mathtt{k},\ 
    a > \int_0^1 (-\log v) \rho(\dd v)
    \text{ and }
    \int_0 v^{-r} \rho(\dd v) < \infty 
    \text{ for some } r>0.
  \]
  Define $\lambda = B - \mathtt{k}$.
  Then, there exists a probability measure $\nu$
  and a random variable $M_\infty$ such that
  for all $x \in (0,\infty)$ and continuous, bounded $f$,
  \[
    e^{-\lambda t} \langle f , \Zb(t) \rangle \to \langle f,\nu\rangle M_\infty, 
  \]
  $\PP_x$-almost surely and in $L^1(\PP_x)$.
\end{theorem}

In fact, $M_\infty$ is the limit of the martingale
$M_t = e^{-(B-\mathtt{k})t} \langle 1, \Zb(t) \rangle$,
and it has expectation constant in
the starting mass of the process: $\EE_x[M_\infty] = 1$.
The $L^1(\PP_x)$ convergence implies in particular \eqref{PF} with $h = 1$,
but in fact, we derive \eqref{eigen} and \eqref{PF},
including a version of \eqref{PF} with exponential rate,
under more general assumptions, as the forthcoming
Propositions~\ref{p:eigen} and~\ref{p:async}. These are fundamental to
our method of proving \autoref{t:slln}.

An equivalent version of this theorem has been shown to hold for a wide variety of branching processes,
including branching diffusions on bounded or unbounded domains \cite{EHK-slln}, superprocesses \cite{EKW-slln}
and, more recently, a general class of branching Markov processes with non-local branching \cite{HHK-snte}.
For a different class of growth-fragmentation processes, a weaker form of this theorem was proven in \cite{BW-lln}, where convergence in $L^1(\PP_x)$ was shown.
In this article, we follow the approaches used in \cite{HHK-snte} and \cite{EHK-slln} to
prove the stronger almost sure convergence.

The second theorem provides a similar result in the transient case.
To aid comparison with the literature, it will be useful to introduce a function
$\kappa$, which in the context of homogeneous growth-fragmentations
is called the `cumulant' and is defined as:
\[
  \kappa(q)
  = aq + \int_{(0,1)} x^q \rho(\dd x) - (B+\mathtt{k}).
\]
The function $\kappa$ is convex and smooth.
We define $q_0$ to be the unique
solution of the equation $\kappa'(q_0) = 0$, or equivalently,
$q_0 = \arginf\kappa$.

\begin{theorem}\label{t:transientlln}
  Assume that
  \[ 
    B > \mathtt{k}, \ 
    a < \int_{(0,1)} (-\log v) \, \rho(\dd v) ,
  \]
  which implies that $q_0> 0$.
  Define $\lambda = \kappa(q_0) = \inf\kappa$, and assume further
  that $\lambda > 0$.
  Let $f \from (0,c]\to\RR$ be continuous with
  $f(x) = O(x^{q_0})$
  as $x\to 0$. Then,
  \[ e^{-\lambda t} \langle f,\Zb(t)\rangle \to 0, \]
  $\PP_x$-almost surely and in $L^1(\PP_x)$.
\end{theorem}

This result can be interestingly
compared with \cite[Theorem 2.3]{Dadoun:agf}, which considers a \emph{homogeneous}
growth-fragmentation $\Zb^\circ$ and takes $f(x) = f_q(x) = x^q$.
Under the assumptions of the above theorem, part of
Dadoun's result is that $e^{-\kappa(q)t} \langle f_q,\Zb^\circ(t)\rangle \to 0$ almost
surely, provided that $q \ge \bar{q}$, where $\bar{q} > q_0$ is the unique solution
of $\bar{q}\kappa(\bar{q}) = \kappa'(\bar{q})$.

The fact that $\bar{q}>q_0$ says, in a quantitative way, that the
cell masses in the growth-fragmentation process with reflection decay to zero
faster than without reflection.
With this in mind, it would be interesting to study the speed of cell masses in more detail,
as done by Dadoun for homogeneous processes in \cite[Corollary 2.4]{Dadoun:agf}, and we leave this
as a future research topic.

We would also like to note several other
possible directions for future work.
Here, we have taken the killing rate to be constant, but it may be possible to relax this restriction
while remaining in a tractable class of processes.
In particular, it would be interesting to consider killing cells, either instantaneously or at
exponential rate, when their mass falls below a certain threshold, to mimic mechanisms
for cell death in biological systems.
An alternative direction would be to take a more general Lévy process as the
underlying trajectory of the 
cell masses; for instance, it should be possible to incorporate Gaussian fluctuations
into this model with little impact on our results.
Finally, as mentioned above,
we have assumed here that cell divisions are binary, but this is far from necessary. Indeed,
provided that the number of cells produced at division is not too large (an `$L \log L$'
condition will suffice), arbitrary division and repartition of mass can be
incorporated into this model; one mathematical approach to this is described in
\cite[p.~8]{BW-lln}. This could prove fruitful in biological models for tumour growth,
since it is known that cancer cells can split into more than two daughter cells. It would also allow
one to study a phenomenon known as aneuploidy, where a cell contains the wrong number of 
chromosomes due to non-binary fission (see \cite{TWD12} and references therein).

The rest of this paper is organised as follows. In \autoref{s:sg} we will analyse the expectation semigroup 
associated with the process $\Zb$, as well as the strongly continuous semigroup associated with
the operator $\mathcal{A}$. We will show that these two semigroups are equivalent, in an appropriate sense, for 
a certain class of functions.
In \autoref{s:gfp} we will formally define the growth-fragmentation process
and consider the basic behaviour of the number of cells in the system.
In \autoref{s:asymp}, we will make use of a spine method to distinguish the
regimes of long-term behaviour.
Finally, sections~\ref{s:lln} and \ref{s:transient} are devoted
to the proofs of Theorems~\ref{t:slln} and \ref{t:transientlln}, respectively.

\section{The growth-fragmentation semigroup}
\label{s:sg}

One way to characterise the growth-fragmentation process is via its expectation semigroup. 
By considering the average behaviour of particles under the action of suitable test functions, 
we characterise as the unique solution to an evolution equation.

On the other hand, from an analytical perspective, it is common to use the generator 
defined in \eqref{GFeq} to define a strongly continuous semigroup, usually
written $(e^{t \mathcal{A}})_{t \ge 0}$, which uniquely solves
the growth-fragmentation equation \eqref{e:gfe-weak}.

In this section, we will explore both perspectives and, in order to consolidate them,
we will show that the two are equivalent for a certain class of test functions. 
Hence, roughly speaking, both the analytical and probabilistic semigroups characterise the
average behaviour of such growth-fragmentation systems, and can therefore both be used
to analyse their long-term behaviour.

We start by considering the expectation semigroup associated with the growth-fragmentation process.
First recall the underlying reflected spectrally negative Lévy process, $\xi^b$, 
obtained by taking the process with Laplace exponent 
\[ 
  \psi_\xi(q) = -\kr + aq - \int_{-\infty}^0 (1-e^{qx})\, \Pi(\dd x), 
\]
satisfying $\EE_0[e^{q\xi_t}] = e^{t\psi_\xi(q)}$,
and reflecting it from above at the level $b$.

Further, recall the growth-fragmentation kernel $\rho$ defined on $(0,1)$ by
$\rho = \Pi\circ g^{-1} + \Pi\circ\bar{g}^{-1}$, where $g(x) = e^x$
and $\bar{g}(x) = 1-e^{x}$.

As before, we write $\Zb = \bigl(\Zb(t), t\ge 0\bigr)$ for the 
Markovian growth-fragmentation associated
with $e^{\xi^b}$, and let $c = e^b>0$, which is the reflection
level for cell sizes in $\Zb$.

Define $\Psi_t f(x) = \EE_x\bigl[ \langle f, \Zb(t)\rangle\bigr]$.
In our first proposition, we derive an evolution equation for which $\Psi_t$
is the unique solution.
We regard $\Psi$ as a semigroup on $L^\infty((0,c])$, the space of
bounded measurable functions on $(0,c]$.

\begin{proposition}\label{p:ev-eqn}
  The semigroup $\Psi$ satisfies the evolution equation
  \begin{align}
    \Psi_t f(x) &= 
    f(x e^{a \cdot t\wedge T_c(x)}) -\kr\int_0^t \dd s \Psi_{t-s}f(xe^{a \cdot s\wedge T_c(x)})\notag \\
    &\quad {} + \int_0^t \dd s \int_0^1 \rho(\dd v) \,
    \bigl[ \Psi_{t-s} f(xe^{a \cdot s\wedge T_c(x)}v) - v \Psi_{t-s} f(xe^{a \cdot s\wedge T_c(x)})\bigr],     
    \label{eveqn}
  \end{align}
  where $T_c(x) = \frac{1}{a} \log\left(\frac{c}{x}\right)$ is the first
  time the deterministic path $t\mapsto xe^{at}$ reaches $c$.
  Moreover, this evolution equation uniquely determines
  $\Psi$ as a semigroup on $L^\infty((0,c])$.
\end{proposition}
\begin{proof}
  Let $\tau_1$ denote the first time that the initial particle either dies or fragments. By considering the two cases $t < \tau_1$ and $t \ge \tau_1$ we have
  \begin{align*}
    \Psi_tf(x) &= f(xe^{a \cdot t\wedge T_c(x)})e^{-{(B+\mathtt{k})t}}\\
    &\qquad {} + (B + \mathtt{k})\int_0^t \dd s \, \frac{B}{B + \mathtt{k}}e^{-(B + \mathtt{k})s}\int_{-\infty}^0 \frac{\Pi(\dd u)}{B}\\
    &\hspace{3cm} {} \times \left\{\Psi_{t-s}f(xe^{a \cdot s\wedge T_c(x)}e^{u}) 
    + \Psi_{t-s}f(xe^{a \cdot s\wedge T_c(x)}(1-e^{u}))\right\}\\
    &= f(xe^{a \cdot t\wedge T_c(x)})e^{-{(B+\mathtt{k})t}}\\
    &\qquad {} + \int_0^t \dd se^{-(B + \mathtt{k})s}\int_{-\infty}^0\Pi(\dd u) \left\{\Psi_{t-s}f(xe^{a \cdot s\wedge T_c(x)}e^{u}) 
    + \Psi_{t-s}f(xe^{a \cdot s\wedge T_c(x)}(1-e^{u}))\right\}.
  \end{align*}
  Making a change of variables $v = e^u$ yields
  \begin{align*}
    \Psi_tf(x) &= f(xe^{a \cdot t\wedge T_c(x)})e^{-{(B+\mathtt{k})t}} \\
    &\quad {} + \int_0^t \dd s 
    e^{-(B + \mathtt{k})s}\int_0^1
    \Pi \circ g^{-1}(\dd v)
    \left\{\Psi_{t-s}f(xe^{a \cdot s\wedge T_c(x)}v) + \Psi_{t-s}f(xe^{a \cdot s\wedge T_c(x)}(1-v))\right\}.
  \end{align*}
  Making another change of variables, $v \mapsto 1-v$, 
  for the second integrand in the final line above,
  and recalling that $\rho = \Pi \circ g^{-1} + \Pi \circ \bar{g}^{-1}$, we have
  \begin{align*}
    \Psi_tf(x) &= f(xe^{a \cdot t\wedge T_c(x)})e^{-{(B+\mathtt{k})t}}\\
    &\quad {} + \int_0^t \dd s  \,
    e^{-(B + \mathtt{k})s}
    \int_0^1 \rho(\dd v)\, 
    \Psi_{t-s}f(xe^{a \cdot s\wedge T_c(x)}v).
  \end{align*}
  Applying Dynkin's identity~\cite[Lemma 1.2, Chapter 4, Part I]{DynPDE} 
  and using the fact that that $\int_0^1 v \, \rho(\dd v) = B$
  shows that $\Psi_t f(x)$ does indeed solve~\eqref{eveqn}.

  \bigskip

  For uniqueness, suppose $\Psi^1_t$ and $\Psi^2_t$ are both solutions to~\eqref{eveqn} and let $\tilde\Psi$ denote their difference. Then, since $B$ is finite, it is straightforward to show that there exists a constant $C > 0$ such that
  \begin{equation*}
    \lvert \tilde\Psi_tf(x) \rvert \le C \int_0^t  \sup_{x \in (0,c]} \lvert \tilde\Psi_{t-s}f(x) \rvert \dd s.
  \end{equation*}
  Gronwall's inequality then yields the result. 
\end{proof}

With this result we have uniquely characterised the average (linear)
behaviour of $\Zb$ using the evolution equation \eqref{eveqn}.
Part of our goal in this work is to describe our probabilistic
results (\autoref{t:slln} and \eqref{PF})
about the growth-fragmentation semigroup in terms of the spectral
theory of the operator $\mathcal{A}$, defined in \eqref{GFeq}. However,
if we were to define $\mathcal{A}$ as the generator of
a semigroup with the same action as $\Psi$
(on the space of continuous functions vanishing at infinity, say),
then its domain would not be large enough to contain the eigenfunction we
identify. To remedy this, we now show that $\mathcal{A}$ 
(or a suitable extension) can be regarded as the
generator of a strongly continuous semigroup on $L^2((0,c])$, which solves
the growth-fragmentation equation \eqref{e:gfe-weak} and agrees with 
the semigroup $\Psi$ above on their common domain of definition.

In general, questions of existence and uniqueness of a semigroup
with generator $\mathcal{A}$ are not straightforward; we refer to
\cite{Davis,Mar-unif,MP-interval} for some
examples of classical and recent approaches.
Our situation is fairly simple since the fragmentation rate
is bounded, but nonetheless it does not seem to be easy to derive from
existing results, so we develop the proof below.

We first focus our attention on the 
growth (i.e., transport) part of the operator. We define
\[
  \mathcal{G}f(x) = axf'(x) - (B + \mathtt{k})f(x), \qquad f \in D_0,
\]
where $D_0$ is the subspace of $L^2((0,c])$
consisting of functions $f\colon (0,c] \to \mathbb{R}$ such that
\begin{itemize}
  \item[(i)] {\color{black}$f$ is smooth with compact support}, and
  \item[(ii)] $\lim_{x \to c}f'(x) = 0$.
\end{itemize}

Then, we have the following proposition.

\begin{proposition}\label{prop:scsemi}
  The operator $\mathcal{G} \colon D_0 \to L^2$ 
  is closable and its closure generates a strongly continuous semigroup on $L^2((0,c])$.
\end{proposition}
\begin{proof}
  The majority of the proof follows the method in \cite{Vidav} so we will only give an outline of the proof here. 
  We start by considering the operators
  \[
    U_t f(x) = e^{-(B + \mathtt{k})t}f(xe^{at}), \quad f \in L^2((0, c]), \, t \ge 0,
  \]
  which we define as $e^{-(B+ \mathtt{k})t}f(c)$ whenever $t > 0$ and $xe^{at} \ge c$. 

  \smallskip

  First note that for each $t \ge 0$, $U_t$ is bounded.
  Also, for any continuous function
  $f\colon (0,c] \to \mathbb{R}$ with compact support, we have
  \[
    \Vert U_t f - f\Vert_2 \to 0, \quad \text{as } t \to 0. 
  \]
  Since the space of compactly supported continuous functions is dense in $L^2$, these observations imply 
  $(U_t)_{t\ge 0}$ is strongly continuous. 
  Hence, its infinitesimal generator is closed and densely defined.

  \smallskip

  Now let $f \in L^2((0,c])$ be such that $t \mapsto f(xe^{at})$ is smooth with compact support,
  and for $h > 0$, set
  \[
    \varphi_h(x) := \int_0^h U_sf(x) \dd s.
  \]
  Due to \cite[p.620]{DunfordNLoP}, $\varphi_h$ lies in the domain of the infinitesimal generator of $(U_t)_{t \ge 0}$. 
  Denote by $\mathcal{G}_1$ the restriction of this generator to functions $\varphi_h$ constructed as above. 
  Then, again, by \cite[p.620]{DunfordNLoP} 
  \[
    \mathcal{G}_1\varphi_h = U_h f - f,
  \]
  and the infinitesimal generator of $(U_t)_{t \ge 0}$ is $\bar{\mathcal{G}_1}$.

  \smallskip

  Our next step is to show $\mathcal{G}$ is closable and its closure equals $\bar{\mathcal{G}_1}$. 
  To show this, we first show that $\varphi_h$ lies in $D_0$. First note that from the definition of $\varphi_h$, it is also smooth with
  compact support.
  Next note that
  \begin{align}
    \varphi_h(xe^{at}) &=  \int_t^{t + h}e^{-(B+ \mathtt{k})(u-t)}f(xe^{au})\dd u \label{step1}.
  \end{align}

  \smallskip

  \noindent Hence,
  \begin{align*}
    \frac{\partial}{\partial t}\varphi_h(xe^{at}) 
    &= e^{-(B+ \mathtt{k})h}f(xe^{a(t+h)}) + (B+\mathtt{k})  \int_0^{t+h}e^{-(B+ \mathtt{k})(u-
    t)}f(xe^{at})\dd u \\
    &\qquad -f(xe^{at}) - (B+\mathtt{k})\int_0^{t}e^{-(B+ \mathtt{k})(u-t)}f(xe^{at})\dd u.
  \end{align*}

  \noindent Letting $t \to 0$, we have
  \[
    \lim_{t \to 0}\frac{\partial}{\partial t}\varphi_h(xe^{at}) = U_h f(x) - f(x) + (B + \mathtt{k})\varphi_h(x),
  \]
  which, in turn, implies that 
  \begin{equation}
    ax\varphi_h'(x) - (B + \mathtt{k})\varphi_h(x) = U_h f(x) - f(x).
    \label{eq1}
  \end{equation}

  \noindent Now, from the definitions of $\varphi_h$ and $U_h$, it is not too difficult to show that
  \[
    \lim_{x \to c}\varphi_h(x) = \frac{f(c)}{B + \mathtt{k}}(1- e^{-(B + \mathtt{k})h})
  \]
  and $\lim_{x \to c}U_hf(x) = e^{-(B + \mathtt{k})h}f(c)$. Taking the limit as $x \to c$ in~\eqref{eq1} 
  and using these two equalities yields $\lim_{x \to c}\varphi_h'(x) = 0$. Hence $\varphi_h \in D_0$.

  \medskip

  The above analysis also shows that
  \[
    \mathcal{G}_1\varphi_h(x) = ax\varphi_h'(x) - (B + \mathtt{k})\varphi_h(x) = \mathcal{G}\varphi_h(x).
  \]
  Since $\varphi_h \in D_0$, it follows that $\mathcal{G}_1 \subset \mathcal{G}$.

  \smallskip

  Finally, for any $\varphi \in D_0$ we have that
  \[
    U_h\varphi - \varphi = \int_0^hU_s\mathcal{G}\varphi \dd s.
  \]
  Due to \cite[Theorem 10.5.2]{Hille}, $\mathcal{G}$ is a restriction of $\bar{\mathcal{G}_1}$. Therefore, $\bar{\mathcal{G}} = \bar{\mathcal{G}_1}$ and so $\mathcal{G}$ is closable and generates the semigroup $(U_t)_{t \ge 0}$.
\end{proof}

Let $\mathcal{F} \colon D_0 \to L^2$ be defined by
\[
  \mathcal{F}f(x) = \int_0^1 f(xy)\rho(\dd y).
\]
Then, we can extend the result of the previous proposition to $\mathcal{G} + \mathcal{F}$.

\begin{corollary}\label{cor:scsemi}
  The conclusion of Proposition~\ref{prop:scsemi} holds when we replace $\mathcal{G}$ by $\mathcal{G} + \mathcal{F}$. 
  Moreover, if we define $\mathcal{A}$ to be the closure of $\mathcal{G} + \mathcal{F}$,
  then its
  domain $\mathcal{D}(\mathcal{A})$ is equal to the domain of the closure of $\mathcal{G}$.
\end{corollary}

\begin{proof}
  First note that since $\mathcal{F}$ it is bounded, it can be extended to an operator with
  domain $L^2((0,c])$. Further, recall $(U_t)_{t \ge 0}$ from the proof of Proposition~\ref{prop:scsemi}.
  We will use Theorem 1 of~\cite{VidavPert} to prove this corollary, which in our context, is as follows. 

  \medskip

  Suppose the following two conditions are satisfied:
  \begin{itemize}
    \item[(I)] $U_t D_0 \subset D_0$ and for all $f \in D_0$, the function $t \mapsto \mathcal{F}U_t f$ is continuous.
    \item[(II)] There exist constants $\alpha > 0$, $\gamma \in [0,1)$ such that for all $f \in D_0$
      \[
        \int_0^\alpha \Vert \mathcal{F}U_tf \Vert \dd t \le \gamma \Vert f \Vert.
      \]
  \end{itemize}
  Then there exists a unique strongly continuous semigroup $(W_t)_{t \ge 0}$ on $L^2((0,c])$ such that 
  \[
    W_tf = U_t f + \int_0^t W_{t-s}\mathcal{F}U_sf \dd s \qquad f \in D_0, \quad t \ge 0.
  \]
  Moreover, the generator of $(W_t)_{t \ge 0}$ is the closure of $(\mathcal{G} + \mathcal{F})\big|_{D_0}$ and its domain is 
  equal to the domain of the closure of $\mathcal{G}$.

  \bigskip

  Now, (I) follows easily from the definitions of $U_t$, $D_0$ and $\mathcal{F}$. For (II), note that
  for any $\alpha>0$,
  \[
    \int_0^\alpha \Vert \mathcal{F}U_tf \Vert \dd t 
    \le 2B \Vert f\Vert  \int_0^\alpha e^{-(B + \mathtt{k})t} \dd t 
    = \frac{2B\Vert f\Vert}{B + \mathtt{k}} \,(1-e^{-\alpha(B + \mathtt{k})}).
  \]
  Choosing $\alpha > 0$ such that 
  $\gamma \coloneqq \frac{2B}{B + \mathtt{k}} \,(1-e^{-\alpha(B + \mathtt{k})}) \in [0,1)$ yields the result.
  Hence, the operator $\mathcal{A}$ with domain $\mathcal{D}(\mathcal{A})$ generates a strongly continuous semigroup, 
  $(W_t)_{t \ge 0}$ on $L^2((0,c])$. Moreover, its domain, $\mathcal{D}(\mathcal{A})$, is equal to that of the domain of the
  closure of $\mathcal{G}$.
\end{proof}

Therefore, the operator $\mathcal{A}$
is the generator of a strongly continuous semigroup $W = (W_t)_{t\ge 0}$ on $L^2((0,c])$
and has $D_0$ as a core.
Fixing $g \in L^2((0,c])$,
if we let $\mu_t$ be defined, though the Riesz representation theorem,
by
$\langle W_tf,g\rangle_2 = \langle f,\mu_t\rangle$
(where $\langle\cdot,\cdot\rangle_2$ represents the
$L^2$ inner product),
then $(\mu_t)_{t\ge 0}$ solves \eqref{e:gfe-weak} with $\mu_0(\dd y) = g(y)\dd y$.

In fact, in a certain sense the semigroups $W$ and $\Psi$ are equivalent:

\begin{proposition}\label{prop:equalsemi}
  For each $f \in L^\infty((0, c])$, $\Vert W_t f - \Psi_t f \Vert_2 = 0$.
\end{proposition}
\begin{proof}
  The proof is identical to that of~\cite[Theorem 8.1]{CHHK-mult}.
\end{proof}

This provides a version of \eqref{e:mu-Psi} in the introduction,
albeit restricted to bounded functions $f$ and initial
conditions having $L^2$ density.
This will let us consolidate the behaviour shown in
\eqref{eigen} and \eqref{PF} later on.

\section{The growth-fragmentation process}
\label{s:gfp}

We now turn our attention to the growth-fragmentation process, $\Zb$.
Although we have already given a construction of $\Zb$ in the introduction, in this section we 
provide another construction via the Ulam-Harris tree, as it will be convenient 
in later sections to refer to the particles via their labels.

To this end, we denote by $\Uu = \bigcup_{n \ge 0}\{0, 1\}^n$
the Ulam-Harris tree of finite sequences of 0s and 1s. 
For $u \in \{0, 1\}^n$, we interpret the length $|u| = n$ as its generation. Further, for $v \in \{0, 1\}$, 
we write $uv \in \{0, 1\}^{n+1}$ to denote the concatenation of $u$ and $v$ and interpret it as
the $v^{th}$ daughter of $u$. The initial individual will be written as $\emptyset$.

Now fix $x \in (0,c]$ and let $\PP_x$ denote the law of the growth-fragmentation process
started from an initial particle at $x_0$. Now define the random variable $\zeta_\emptyset$ by
\[
  \PP_x(\zeta_\emptyset > t) = e^{-(B + \texttt{k})t}, \quad t \ge 0,
\]
which we see as the minimum of the fragmentation time and death time of the initial particle. 
Setting $b_\emptyset = 0$ and $d_\emptyset = \zeta_\emptyset$, it follows that the trajectory of the
initial particle is given by
\[
  Z_\emptyset(t) = xe^{a \cdot t \wedge T_c(x)}, \quad t \in [b_\emptyset, d_\emptyset),
\]
where we recall that $T_c(x) = \frac{1}{a}\log(c/x)$.
With probability $\frac{\mathtt{k}}{B+\mathtt{k}}$, $\zeta_\emptyset$ denotes a killing
time of $\emptyset$, in which case the process stops.
If not, then $\zeta_\emptyset$ denotes the first fragmentation time,
and then we set $b_u = d_\emptyset$ for each daughter $u \in \{0, 1\}$ of the initial 
particle and randomly choose $v \in (0, 1)$ according to 
the probability measure $\frac{\rho(\dd v)}{2B}$.
The masses of the daughters are then given by $xv$ and $x(1-v)$. The lifetime of each daughter is denoted by $\zeta_u$, which has the same distribution as $\zeta_\emptyset$ under $\PP_{xv}$ or $\PP_{x(1-v)}$. We also set $d_u = b_u + \zeta_u$ and the trajectory of the $u^{th}$ daughter is given by
\[
  Z_u(t) = ye^{a \cdot t \wedge T_c(y)}, \quad t \in [b_u, d_u),
\]
for $y = xv, x(1-v)$. 
The process then continues iteratively, 
with each cell $u$ and the initial mass of its daughters $u0$ and $u1$ 
being determined in the same way, and so on; this
gives a trajectory $(Z_u(t))_{t \in [b_u, d_u)}$ for each $u \in \Uu$. Setting $\Uu_t$ to be the set of individuals alive at time $t$, i.e. $\Uu_t = \{u \in \Uu : t \in [b_u, d_u)\}$, the growth-fragmentation process at time $t$ is given by the collection of atomic measures
\[
  \Zb(t) = \sum_{u \in \Uu_t}\delta_{Z_u(t)}.
\]

In the introduction and in the previous section, this process was constructed in a different way,
using the `cell process' $\xi$ and adding new cells at each jump. These two approaches
are consistent: if $(u_1,u_2,\dotsc)$ is a sequence of cell labels with the property
that, for each $n\ge 1$, $Z_{u_{n+1}}(b_{u_{n+1}})/Z_{u_n}(d_{u_n}-) \ge 1/2$, then
\[ \xi_t = Z_{u_n}(t), \text{ where } b_{u_n}\le t < d_{u_n}, \qquad t \ge 0, \]
is the cell process.

\bigskip

Let us now briefly turn our attention to the growth of the total number
of cells in $\Zb$.
Set $\lambda_* = B-\kr$ and define $N_t = \langle 1,\Zb(t)\rangle$, 
the number of cells in the growth-fragmentation
at time $t$. Since the branching and killing rates are spatially
independent, $N$ is a discrete-space, continuous-time branching process.

Recall that a branching process 
is subcritical, critical or supercritical according to whether
$\EE[N_t]$ is negative, zero or positive,
and that in the subcritical and critical cases, the process
becomes extinct (eventually $N_t = 0$) with probability $1$.
We have the following simple consequence of branching process theory:

\begin{proposition}\label{p:N-crit}
  $N$ is subcritical, critical or supercritical according to whether
  $\lambda_*$ is negative, zero or positive.
  In the supercritical case,
  the extinction probability is 
  $\frac{\kr}{B}$.
\end{proposition}
\begin{proof}
  This follows using classical branching process results
  \cite[\S I.5]{AN-bp}.
  The probability generating function of the offspring distribution
  is $g(z) = \frac{\kr + Bz^2}{B+\kr}$, and
  equation $g(z) = z$ has roots $1$ and $\frac{\kr}{B}$,
  which implies the extinction probability in the result.
\end{proof}

\medskip

We end this section with a short discussion of the relationship between the various parameters 
of the growth-fragmentation process. For the particular model we consider in this paper, 
all of our results can be expressed explicitly in terms of the parameters $a, \kr$ and $\rho$,
due to the fact that $\mass{\rho} = 2\mass{\Pi} = 2B$,
where $\mass{\cdot}$ indicates the total mass of a measure.
We emphasise that the factor of two
appears here due to the fact that two particles are produced at every fragmentation event.

In a more general model, where we allow particles to fragment into more than two pieces,
the total mass of $\rho$ is given by the fragmentation rate multiplied by the average number
of offspring, and $\lambda_*$ is replaced by $\mass{\rho}-B-\mathtt{k}$. Moreover, 
the results of this paper will still hold but with a slight adjustment of the conditions; 
the condition $B > \mathtt{k}$ in \autoref{t:slln} would become $\mass{\rho} -B-\mathtt{k} > 0$, 
and an `$L \log L$' condition on the number of offspring would be required, 
which is redundant in the binary case.
One key difference in this more general model is that $\rho$ 
(or equivalently, the linear evolution equation) no longer uniquely determines
the growth-fragmentation process.
For this, one needs to work with a non-linear version of the
growth-fragmentation equation, as in \autoref{s:lln}.

\section{Asymptotic regimes}
\label{s:asymp}

In the previous section, we determined conditions for the total number of cells to
either reach zero with probability 1 (the (sub)critical cases) or to survive
with positive probability (the supercritical case).
We now consider the long-term behaviour of cell masses,
with the goal of demonstrating that these settle into an equilibrium in the long-term;
this will be most important in the supercritical case, where it will be the foundation
of \autoref{t:slln}, but the results of this section apply regardless
of (sub- or super-)criticality.

Our main tool will be the analysis of a single \emph{tagged cell}, whose behaviour
will tell us something about the average behaviour of the process $\Zb$.
We can think of this as following a distinguished line of descent in the cell line,
where at each cell division event we uniformly pick a daughter cell to follow.

The motion of the cell mass of this tagged cell will be determined by an exponential
reflected Lévy process, much like the `cell process' $\xi$ in the introduction;
however, the process itself is different, since it should represent not the
locally largest cell, but a typical cell.

Introduce a Lévy process $\eta$ with
drift $a$ and
Lévy measure $\Pi_\eta = \rho\circ (g^{-1})^{-1}$, where $g^{-1}(x) = \log x$.
Equivalently, this Lévy measure can be expressed in terms of that of $\xi$
by
$\Pi_\eta = \Pi + \Pi\circ s^{-1}$, where $s(x) = \log(1-e^{x})$.
(Hence, if $\Pi$ has density $\pi$, then $\Pi_\eta$
has density $\pi_\eta(x) = \pi(x) + \pi(\log(1-e^{x}))e^x(1-e^x)^{-1}$.)
Denote by $\psi_\eta$ its Laplace exponent, which can be expressed as
\[
  \psi_\eta(q) = aq + \int_{(-\infty,0)} (e^{qx}-1)\, \Pi_\eta(\dd x) 
  = aq + \int_{(0,1)} (v^q-1) \, \rho(\dd v).
\]
Again, as with $\xi$, we define $\eta^b$ 
to be this process reflected at $b$,
i.e., $\eta^b_t = \eta_t - (\bar{\eta}_t-b)\vee 0$.

The mass of the tagged cell is given by $Y = \exp(\eta^b)$.
Our goal for the remainder of this section is the investigate
the long-term  behaviour of $Y$, identify the growth-fragmentation semigroup
in terms of the semigroup of $Y$, and hence prove the asynchronous
exponential growth \eqref{PF}. This will be a key building block in
the proof of the strong law, \autoref{t:slln}.

In past work on growth-fragmentation equations \cite{BW-gf-fk,cavalli2019},
a key element of the analysis was an auxiliary function defined in terms
of return times of $Y$. In our context, we may define this as follows.
Let $S(y) = \inf\{ t > 0 : Y_t \ne y\}$, and let
$H(y) = \inf\{ t > S(y): Y_t = y \}$ be the first return time to $y$.
Define
\[ L_{x,y}(q) = \EE_x[ e^{-(q-\lambda_*)H(y)}; H(y)<\infty]. \]
Let $q_* = \inf\{q \in \RR : L_{c,c}(q) < \infty\}$ and
$\lambda = \inf\{q \in \RR : L_{c,c}(q) < 1\}$.
The quantity $\lambda$ is the conjectural exponential rate in \eqref{PF},
and in previous works, the behaviour of $L$ was used in order to deduce
asymptotic results for the growth-fragmentation semigroup. Our situation
is actually a little simpler than \cite{BW-gf-fk}, and we can
avoid this step, but we provide results on $L$ for the sake of comparison.

Our first result pertains to the transience and recurrence of $Y$.
Roughly speaking, the classification
depends on the mean of the underlying L\'evy process, $\eta$. 
We distinguish four possible cases:
\begin{align}
  a &< \int_0^1 (-\log v) \rho(\dd v), \tag{T}\label{T} \\
  a &= \int_0^1 (-\log v) \rho(\dd v), \tag{NR}\label{NR} \\
  a &> \int_0^1 (-\log v) \rho(\dd v), \tag{PR}\label{PR} \\
  a &> \int_0^1 (-\log v) \rho(\dd v)
  \text{ and }
  \int_0 v^{-r} \rho(\dd v) < \infty 
  \text{ for some } r>0.
  \tag{ER}\label{ER}
\end{align}

\begin{proposition}
  \label{p:asym}
  For $t \ge 0$ and $x \in (0,c]$, the many-to-one formula
  \[ \Psi_t f(x) = \EE_x[ f(Y_t) e^{\lambda_* t}] \]
  holds.
  Further, defining $q_* = \lambda_* + \max\{\inf\psi_\eta, -\mass{\Pi_\eta}\}$, we have the following cases.
  \begin{enumerate}
    \item
      If \eqref{T} holds, then $Y$ is transient, $\lambda = q_* = \lambda_* + \inf \psi_\eta$ and $L_{c,c}(\lambda)<1$.
    \item
      If \eqref{NR} holds, then $Y$ is null recurrent, $\lambda = \lambda_*$, $L_{c,c}(\lambda) = 1$ and
      \newline $L'_{c,c}(\lambda) = -\infty$.
    \item
      If \eqref{PR} holds, then $Y$ is positive recurrent, $\lambda = \lambda_*$, $L_{c,c}(\lambda) = 1$ and 
      \newline $L'_{c,c}(\lambda) > -\infty$.
    \item
      If \eqref{ER} holds, then for any $0 < w < r$,
      there exist $k>0$ and $C>0$
      such that for all $x\in (0,c]$, $t\ge 0$,
      and $f$ such that $f_w(x) \coloneqq (x/c)^{w} f(x)$
      is continuous and bounded,
      \[
        \bigl\lvert \EE_x[ f(Y_t)] - \langle f,\nu\rangle \bigr\rvert
        \le 
        \lVert f_w \rVert \bigl( (c/x)^{w} + C \bigr) e^{-kt},
      \]
      where $\nu$ is the invariant distribution of $Y$.
      Indeed, we can take $k = -\psi_\eta(-w)$ and $C = \int (c/x)^w \, \nu(\dd x) = \frac{-\psi_\eta'(0)w}{\psi_\eta(-w)}$.
  \end{enumerate}
\end{proposition}

\begin{proof}
  The proof of the many-to-one formula, $\Psi_t f(x) = \EE_x[ f(Y_t) e^{\lambda_* t}]$, follows in a similar way to the 
  proof of Proposition~\ref{p:ev-eqn},
  and we omit the full detail.
  Indeed, splitting on the first time $\eta$ is either killed or jumps, 
  applying Dynkin's identity and some simple algebraic manipulations shows that $\EE_x[ f(Y_t) e^{\lambda_* t}]$ 
  also solves~\eqref{eveqn}. 
  Since solutions of \eqref{eveqn} are unique,
  it is equal to $\Psi_t f(x)$ for each $x \in (0,c]$ and each $t \ge 0$. 

  Next we deal with the recurrence and transience of the process $Y$ in the different cases. 
  We will show that $Y$ is transient when \eqref{T} holds and that $Y$ is recurrent
  when either \eqref{NR} or \eqref{PR} hold, and defer the proof of the distinction between
  positive and null recurrence to \autoref{p:nu}.

  Differentiating $\psi_\eta$ and letting $q \to 0$, we have
  \begin{equation}\label{e:mean}
    \psi_\eta'(0) = a + \int_{(0,1)}\log v\,\rho(\dd v)
  \end{equation}
  The recurrence or transience of $Y$ is equivalent to the recurrence or transience
  of the reflected Lévy process $\eta^b$. Considering its paths, we observe that
  $\eta^b$ is recurrent if and only if $\limsup_{t\to\infty}\eta_t = \infty$ a.s.;
  and by
  \cite[Theorem 7.1]{Kyp2},
  this occurs if and only if $\psi_\eta'(0)\ge 0$.
  Combining this with \eqref{e:mean},
  we see that $\eta^b$, and hence $Y$, is recurrent if and only if \eqref{PR} or \eqref{NR} holds,
  and transient if \eqref{T} holds.

  \smallskip

  We now consider the values of $\lambda$ and $L_{c, c}(\lambda)$ for each of the first three cases.
  By conditioning on the first time $Y$ jumps away from $c$, it is straightforward to show that
  \begin{equation}
    L_{c,c}(q) = 
    \begin{cases}
      1-\dfrac{a\Phi(q-\lambda_*)}{\Vert\Pi_\eta\Vert + q - \lambda_*}, & \text{if } q > q_* \\
      \infty & \text{if } q \le q_*,
    \end{cases}
  \end{equation}
  where $\Phi$ is the right inverse of $\psi_\eta$.
  Since $q_* \ge -\mass{\Pi_\eta}$, it is clear that the denominator
  is positive whenever $q > q_*$. 
  Hence, $L_{c,c}(q) = 1$ for some $q \ge q_*$ if
  and only if there is a solution to the equation
  $\Phi(q-\lambda_*) = 0$ with $q \ge q_*$.

  If $\psi_\eta'(0)\ge 0$, then $\Phi(0) = 0$ and so
  $q = \lambda_*$ gives us the solution we seek.
  It also follows that $L_{c, c}(\lambda) = \PP_{c}(H(c) < \infty) = 1$
  since $\psi'(0) \ge 0$ implies that $Y$ is recurrent. 

  If $\psi_\eta'(0) < 0$, 
  then $\Phi(q-\lambda_*)$ is strictly positive for all
  $q\ge q_*$, and so $\lambda = q_*$
  and $L_{c, c}(\lambda) < 1$.
  To show that $q_* = \lambda_* + \inf\psi_\eta$,
  we argue as follows:
  \[ 
    \psi_\eta(q) + \mass{\Pi_\eta} = aq + \int_{-\infty}^0 e^{qx}\Pi_\eta(\dd x) \ge 0, 
  \]
  and the two summands on the right are positive and do not approach zero simultaneously
  (if $q\ge 0$). Hence $\inf\psi_\eta = \inf_{q\ge 0} \psi_\eta(q) > -\mass{\Pi_\eta}$.

  \smallskip

  To prove the statements about the derivative of $L$, note that
  \[
    L'_{c,c}(q) 
    = -\frac{a\Phi'(q-\lambda_*)(q - \lambda_* + \mass{\Pi_\eta}) - a\Phi(q - \lambda_*)}
    {(q - \lambda_* + \mass{\Pi_\eta})^2},
  \]
  and observe that $\Phi'(q-\lambda_*) = \infty$ if and only if $q = \lambda_*+\inf\psi_\eta$.
  Since $\inf\psi_\eta = 0$ if and only if \eqref{NR} holds, we obtain that
  $L'_{c,c}(\lambda) = -\infty$ if and only if we are in case \eqref{NR}, which
  was the claim.

  \smallskip

  Part (iv) follows by Theorem 2 of \citet{GS-exp}.

  This completes the proof.
\end{proof}

\begin{remark}
  \begin{enumerate}
    \item
      As remarked in the proof, the cases \eqref{T}, \eqref{NR} and \eqref{PR}
      are equivalent to
      $\EE[\eta_1]$ being negative, zero and positive respectively,
      and the moment condition in \eqref{ER} is equivalent to
      $\EE[e^{r\eta_1}]<\infty$.
    \item
      In previous works, the value of $L_{c,c}(\lambda)$ was used to
      determine the transience or recurrence of $Y$. We do not need this
      for this result, 
      but we continue the spirit of these works by making heavy use of
      the supermartingale function $\ell(x) = L_{x,c}(\lambda)$
      in our analysis of the case \eqref{T} in section~\ref{s:transient}.
  \end{enumerate}
\end{remark}

\subsection{Recurrent regimes}

We now consider the cases \eqref{NR}, \eqref{PR} and \eqref{ER} in more detail.
Our goal is to find an invariant measure for the process $Y$, which will
turn out to be an eigenmeasure of the growth-fragmentation operator.

The methods in the section come from the theory of spectrally negative Lévy processes,
of which $\eta$ is an example, and we begin by defining the \emph{scale function}
$W \from [0,\infty) \to [0,\infty)$ with Laplace transform
\[
  \int_0^\infty e^{-\beta x} W(x) \, \dd x = \frac{1}{\psi_{\eta}(\beta)},
\]
valid for $\beta > \Phi(0)$, where
$\Phi(q) = \sup\{ p : \psi_\eta(p) \ge q \}$.
We will also use $W$ to denote the (Stieltjes) measure associated with the
scale function.

For details of the existence and properties of scale functions, we refer to
\cite[§8]{Kyp2}.
One property that will be useful below is that $W$ has a left-derivative
on $(0,\infty)$, which we will write as $W'$.

Having recalled some of the theory of scale functions, we are in a position
to identify the invariant measure we seek.

\begin{proposition}\label{p:nu}
  Suppose we are in the case \eqref{PR} or \eqref{NR}.
  Let $m$ be the pushforward of the measure $W(\dd x)$ by 
  the function $x \mapsto ce^{-x}$,
  that is,
  \[
    m(\dd y)
    = \frac{1}{a}\delta_c(\dd y)
    + W'(\log(c/y))\frac{\dd y}{y}.
  \]

  Then, $\EE_c\left[\int_0^{H(c)} f(Y_s)\,\dd s\right] = \frac{a}{2B} \langle f,m\rangle$.

  In case \eqref{PR}, $\langle 1,m \rangle = \left(a + \int_0^1 \log s\, \rho(\dd s)\right)^{-1}$,
  and $\nu = m/\langle 1,m\rangle$ is the invariant distribution of $Y$.

  In case \eqref{NR}, $\langle 1,m\rangle =\infty$, and $m$ is a (unique up to a multiplicative
  constant) $\sigma$-finite invariant measure of $Y$.
  In the latter case, define $\nu = m$.
\end{proposition}
\begin{proof}
  We give quite a general proof, which works under assumptions
  \eqref{PR} and \eqref{NR}, and even under rather general assumptions about
  $Y$ (or $\eta)$.
  In the case \eqref{PR} and under our running assumption that $\eta$
  has no Gaussian part and finite jump rate, a simpler (if not shorter) proof
  is available; see \autoref{r:nu}\ref{r:nu:potential}.

  Let $R_t = \bar{\eta}_t - \eta_t$, the process
  $\eta$ reflected in its running maximum, and define
  $L^{-1}_t \coloneqq \inf\{s>0 : \bar{\eta} > t\}$ to be the inverse local time of $\eta$
  at the maximum.
  Let $\PP = \PP_c$.
  Under $\PP$, we have $Y_t = e^{b-R_t}$, so $L^{-1}$
  can be interpreted either as the inverse local time of $R$ at $0$ or
  the inverse local time of $Y$ at $c$.

  For arbitrary $f$, define $g(x) = f(ce^{-x})$. Following the proof of \cite[Theorem VI.20]{Ber-Levy},
  we obtain
  \begin{align*}
    \EE\int_0^{L^{-1}_t} f(Y_s)\,\dd s
    = \EE\int_0^{L^{-1}_t} g(R_s)\, \dd s 
    &= \EE\int_0^\infty g(R_s) \Indic{\bar{\eta}_s < t} \, \dd s \\
    &= k \int_0^\infty \hat{U}(\dd z) g(z) \int_0^t U(\dd y) \\
    &= k\delta^{-1} t \langle g,\hat{U}\rangle ,
  \end{align*}
  where $U$ and $\hat{U}$ are the ascending and descending renewal measures of $\eta$,
  $\delta$ is the drift coefficient of the upward ladder height process and $k$
  is some constant.

  Now we note that $\hat{U}(\dd z) = k'e^{-\Phi(0)z}W(\dd z)$,
  where $k'$ is another constant \cite[p.~196, equation (4)]{Ber-Levy}.
  Since we are in the case where $Y$ is recurrent
  (and hence $\eta$ does not drift to $-\infty$), $\Phi(0) = 0$.

  In conclusion, we have shown that
  $\EE\int_0^{L^{-1}_t} f(Y_s)\, \dd s = k'' t \langle f,m \rangle$,
  with $m$ as in the proposition and $k''$ a constant.
  Since $\PP_c(H(c)<\infty) = 1$, \cite[§XIX.3.46]{DMM-pp} implies that
  $m$ is an invariant measure for $Y$.

  The fact that $1/\langle 1,m\rangle = \EE_0[\eta_1]$ follows from
  the Laplace transform of the measure $W$ appearing on p.~238 of \cite{Kyp2}.

  Finally, we deduce the claim about $\EE_c\int_0^{H(c)} f(Y_s)\,\dd s$.
  It is shown in \cite[Theorem 8.1]{Get-exc} that this measure is invariant
  for $Y$ (and indeed, this result is at its heart the same as that of
  \cite{DMM-pp} cited above). Since the invariant measure is unique
  up to normalisation, it suffices to compute the two measures of interest
  for some simple function in order to determine the correct
  normalisation.

  Let $f = \Ind_{\{c\}}$. On the one hand, we have
  $\langle f,m\rangle = \frac{1}{a}$.
  On the other, if we define $T_1 = \inf\{ t \ge 0 : Y_t < c\}$, we have
  \[
    \EE_c \int_0^{H(c)} f(Y_s) \, \dd s
    = \EE_c[T_1]
    = \frac{1}{\mass{\Pi_\eta}},
  \]
  and, recalling that $\mass{\Pi_\eta} = 2B$, this proves the
  desired normalisation.
\end{proof}

\begin{remark}\label{r:nu}
  \begin{enumerate}
    \item\label{r:nu:potential}
      It is also possible to obtain the 
      formula for $m$ by considering the first time $\eta$ exits an interval $[-x, b]$ for $x > 0$, using the same
      method as \cite[Theorem~8.7]{Kyp2}, and then letting $x \to \infty$.
      Then, the invariance of $m$ follows, under assumption 
      \eqref{PR}, from \cite[VI.1.5]{Asm-apq2}.

    \item
      We also note that $m$ is still an invariant measure when $\xi$ (or $\eta$) is replaced by
      an arbitrary spectrally negative Lévy process, even if it has a Gaussian part
      or infinite jump rate; however, then $m$ should be interpreted
      in terms of the excursions of $Y$ from $c$.

    \item
      The measure $W$ has Laplace transform
      $\beta \mapsto \frac{\beta}{\psi_\eta(\beta)}$.
      Under assumption \eqref{PR},
      since $\nu$ is the unique invariant probability
      measure for $Y$, the previous result is equivalent,
      by Laplace inversion, to \cite[Corollary IX.3.4]{Asm-apq2},
      which gives another proof even under the general Lévy assumption.
      However, we
      are not aware of an existing proof of this result in the null recurrent case.
  \end{enumerate}
\end{remark}

We are now in a position to solve the eigenproblem \eqref{eigen}
and state the asynchronous exponential growth behaviour \eqref{PF},
which are interesting results in their own right, as well as stepping
stones in our proof of the strong law of large numbers.
It is worth pointing out that, whereas our main strong law result,
\autoref{t:slln}, is stated under the assumption that $\lambda_*>0$
and \eqref{ER} holds, these weaker results hold under more general assumptions,
and in particular make no requirement on the growth or decay
of the number of particles.

\begin{proposition}\label{p:eigen}
  Assume \eqref{NR} or \eqref{PR}.
  Let $h \equiv 1$ and $\nu$ be as defined above.
  Then:
  \begin{enumerate}
    \item
      $e^{\lambda t}$ is the spectral radius of $\Psi_t$.
    \item
      $h \in \mathcal{D}(\mathcal{A})$ and
      $\mathcal{A}h = \lambda h$.
    \item
      For every $f \in \mathcal{D}(\mathcal{A})\cap L^\infty(0,c]$,
      $\nu\mathcal{A}f = \lambda \nu f$.
  \end{enumerate}
\end{proposition}

\begin{proof}
  \begin{enumerate}
    \item
      This follows easily from the fact that $\lambda = \frac{1}{t}\log \Psi_t f$, for every bounded function $f$.
    \item
      The first claim follows from \autoref{cor:scsemi} and the second from the definition of $\lambda$.
    \item
      First note that from the proof of \autoref{prop:scsemi} and ~\autoref{cor:scsemi}, 
      the set of function $D_0$ is 
      a core for $(\mathcal{A}, \mathcal{D}(\mathcal{A}))$. 
      Hence, we first show that this claim holds for $f \in D_0 \cap L^\infty(0,c]$.

      \smallskip

      Due to \autoref{prop:equalsemi} and the fact that $\nu$ is the left eigenmeasure 
      for $\Psi_t$ corresponding to 
      the eigenvalue $e^{\lambda t}$, we have $\nu\Psi_t f = \nu W_t f$. 
      Then, due to \autoref{cor:scsemi}, the claim follows
      for $f \in D_0 \cap L^\infty(0,c]$.

      \smallskip

      Now choose $f \in \mathcal{D}(\mathcal{A}) \cap L^\infty(0,c]$. Since $D_0$ is a core of
      $\mathcal{A}$, 
      this means
      we can choose a sequence of functions $(f_n)_{n \ge 1}\subset D_0$ such that $\Vert f_n - f \Vert _\mathcal{A} \to 0$ as 
      $n \to \infty$, where $\Vert \cdot \Vert_\mathcal{A}$ denotes the graph norm. In particular, this means that both $f_n \to f$ and 
      $\mathcal{A}f_n \to \mathcal{A}f$
      in $L^2((0,c])$ as $n \to \infty$, and hence the result follows.
  \end{enumerate}
\end{proof}

\begin{proposition}\label{p:async}
  Assume \eqref{PR}.
  For all $f \in C_b((0,c])$,
  and all $x \in (0,c]$,
  \[ \lim_{t\to\infty} e^{-\lambda t} \Psi_t f(x) = \langle f, \nu \rangle. \]
  If \eqref{ER} holds, then
  there exists $k>0$ such that for all $x\in (0,c]$ and
  $f \in C_b((0,c])$,
  \[
    \lim_{t \to \infty} e^{k t} \bigl\lvert e^{-\lambda t} \Psi_t f(x) 
    - \langle f, \nu \rangle\bigr\rvert = 0.
  \]
\end{proposition}
\begin{proof}
  This is a corollary of \autoref{p:asym} and \autoref{p:nu}.
\end{proof}

The two preceding propositions prove 
equations \eqref{eigen} and \eqref{PF} in the introduction;
that is, they demonstrate the asynchronous exponential growth
of the expectation semigroup, and its connection with the
spectral theory of the operator $\mathcal{A}$.

\section{Law of large numbers}
\label{s:lln}

This section is dedicated to the proof of \autoref{t:slln}, so
we assume the process
is supercritical (i.e., $\lambda_* = B - \mathtt{k} > 0$)
and $Y$ is exponentially recurrent.
Recall that in this case $\lambda = \lambda_*$.

Let
\[ M_t = e^{-\lambda t} N_t, \]
where $N_t$ is the number of cells alive at time $t$. Then, $M$ is a martingale; indeed, it
is the intrinsic martingale in a continuous-time branching process (regardless of $x$).
By \cite[§III.4, p.~108]{AN-bp}, $\EE_x[M_t^2] = 2(1-e^{-\lambda t})$, so
$M$ is $L^2$-bounded, and hence has a limit $M_\infty$ almost surely and in $L^1$. 
Note that Doob's martingale inequality also yields convergence in $L^2$.

\subsection{Skeleton decomposition}

A key element in the proof is a \emph{skeleton decomposition} for the 
growth-fragmentation process.
This corresponds to splitting $\Zb$ into a tree of particles that 
survive forever, dressed with trees of particles that all die out. 

In order to describe this precisely, we start by specifying the
non-linear semigroup of the growth-fragmentation process; this
describes the behaviour of all the particles, rather than
the average characterised by $\Psi_t$.

Recall $\Uu_t$, the set of labels of particles alive at time $t$, 
and for measurable functions $f$ with $\mass{f}_\infty < 1$, define the non-linear semigroup
\[
  u_t[f](x) \coloneqq \EE_x\left[\prod_{u\in \Uu_t}f({Z_u(t)})\right], \quad t \ge 0,\, x \in (0, c].
\]

\begin{proposition}
  The semigroup $(u_t)_{t \ge 0}$ satisfies the following non-linear growth-fragmentation equation
  \begin{equation}
    u_t[f](x) = f(x e^{a \cdot t\wedge T_c(x)}) + \int_0^t  G[u_{t-s}[f]](x e^{a\cdot s\wedge T_c(x)}) \dd s,
    \label{nonlinear}
  \end{equation}
  where 
  \begin{align}
    G[f](x) &= (B + \kr)\left(\frac{B}{B + \kr}\int_0^1\dd v\kappa(v)f(xv)f(x(1-v)) + \frac{\kr}{B+\kr} - f(x) \right)
    \label{NLGFop}
  \end{align}
  and $\kappa(v) = \pi(v)/vB$.
\end{proposition}

\begin{proof}
  Conditioning on the time of the first event (fragmentation or killing), we have  
  \begin{align*}
    u_t[f](x) &= f(x e^{a \cdot t\wedge T_c(x)})e^{-(B + \kr)t} + \int_0^t \dd s (B + \kr)e^{-(B + \kr)s}\notag\\
    &\quad \times \left( \frac{B}{B + \kr}
    \int_{-\infty}^0\frac{\Pi({\dd}u)}{B}u_{t-s}[f](e^{a \cdot s\wedge T_c(x)}e^u)u_{t-s}[f](e^{a \cdot s\wedge T_c(x)}(1-e^u))+ \frac{\kr}{B+\kr} \right).
  \end{align*}
  Following the same steps as the linear case, namely making the change of variables $v = e^u$ and applying Dynkin's lemma, we obtain the required result.
\end{proof}

\noindent For ease of notation, we will henceforth write
\[
  G[f](x) = (B + \kr)\mathcal{E} \left[ f(xV)f(x(1-V))\mathbf{1}_{N = 2} + \mathbf{1}_{N=0} - f(x)\right],
\]
where under $\mathcal{P}$ with expectation $\mathcal{E}$,
$N \in \{0, 2\}$ is the number of particles produced at
the first event (fragmentation or killing) with 
$\mathcal{P}(N = 2) = B/(B + \kr)$, 
and $V$ is chosen according to $\kappa$.

\bigskip
\noindent 
We turn now to the skeleton decomposition.
Let $C_u$ denote the `colour' of cell $u$.
We write $C_u = b$ if the descendants of cell $u$ survive forever (a `blue' particle)
and $C_u = r$ if the descendants of cell $u$ eventually become extinct (a `red' particle).
Let $\mathbf{C}(t) = (C_u : u \in \Uu_t)$.

Define $\zeta \coloneqq \inf\{t \ge 0 : N_t = 0\}$ to be the 
lifetime of the process,
$p \coloneqq \mathbb{P}_x(\zeta =  \infty)$ to be the survival probability 
and $w = 1-p$ to be the extinction probability.
Note that since the fragmentation kernel and rate do not 
depend on the particle size, $p$ and $w$ do not depend on 
$x \in (0,c]$. By \autoref{p:N-crit},  
$w = \frac{\kr}{B} < 1$.

The following result about the colouring will be extremely useful in what follows:
\begin{lemma}\label{l:colouring}
  Given $\FF_t$, the colours of the cells of
  $\Zb(t)$ are given by independently choosing each cell to be blue
  with probability $p$ and red with probability $w$.
\end{lemma}
\begin{proof}
  This is shown in \cite{HHK-snte}. The independence of the colours
  is due to the spatial homogeneity of the branching and killing rates.
\end{proof}

We denote by $(\Zb,\mathbf{C}) = ((\Zb(t), \mathbf{C}(t)) : t\ge 0)$
the process
of cells marked by their colours.

\bigskip

We first describe the `red trees'. For $A \in \FF_t$ define
\begin{equation}
  \PP_x^R(A) \coloneqq \PP_x(A \mid C_\emptyset = r).
\end{equation}

\begin{proposition}[Red trees]
  The dynamics of the process $(\Zb, \PP_x^R)$ can be described as follows. From a single particle at position $x \in (0, c]$, the particle will grow according to $s \mapsto xe^{a \cdot s\wedge T_c(x)}$, as before. Then, for $x \in (0, c]$ and bounded, measurable $f$, the branching generator is given by
  \begin{align}
    G^R[f](x)
    &= \frac{1}{w}\left(G[fw](x) - f(x)G[w](x) \right).
    \label{redfrag}
    \\
    &= (B+\kr) \mathcal{E}\left [ \frac{\kr}{B+\kr} f(xV)f(x(1-V))
    + \frac{B}{B+\kr} - f(x)\right ]
    \nonumber
  \end{align}
\end{proposition}

The proof of this proposition can be found in~\cite[Prop 2.1]{HHK-snte}. 
We now consider the case where we condition on $C_\emptyset = b$.
In order to do so, we first define a process $(\Zb, \PP^B)$, which should be
thought of as a process representing the blue cells only.
Started from an initial cell of size $x \in (0,c]$, the cell increases in size 
according to $s \mapsto xe^{a \cdot s \wedge T_c(x)}$ and will fragment according to the following operator,
\begin{align}
  G^B[f](x) 
  &= \frac{1}{p}(G[pf + w] - (1-f)G[w])
  \label{bluefrag}
  \\
  &=
  B(p+2w)
  \mathcal{E}\left[
    \frac{p}{p+2w} f(xV)f(x(1-V)
    + \frac{w}{p+2w}\bigl( f(xV) + f(x(1-V))\bigr)
    - f(x)
  \right]
  \nonumber
\end{align}
Note $B(p+2w) = B+\kr$.

We are now ready describe the skeleton decomposition. 

\begin{proposition}[Skeleton decomposition]
  \label{p:skele}
  With probability $w$, $(\Zb, \mathbf{C}, \PP_x)$ is equal in law to $(\Zb, \PP_x^R)$
  with all cells coloured red;
  and with probability $p$, $(\Zb, \mathbf{C}, \PP_x)$ is equal in law to 
  $(\Zb, \mathbf{C}, \PP_x( \cdot \mid C_\emptyset = b))$, 
  which, in turn, is equal in law to 
  a process $(\Zb, \PP_x^B)$ coloured blue,
  dressed with copies of $(\Zb, \PP_x^R)$ coloured red. 
  In the latter case, the joint branching of the blue tree and dressing with 
  red trees can be described via the following generator:
  \begin{align}
    H[f, g](x)
    &=
    B(2w+p)
    \mathcal{E}[
    \frac{p}{2w+p} f(xV)f(x(1-V))
    \nonumber\\
    &\quad {}
    +\frac{w}{2w+p} f(xV)g(x(1-V)) + \frac{w}{2w+p} g(xV)f(x(1-V)
    - f(x)
    ]
    \label{jointfrag}
  \end{align}
\end{proposition}

The idea behind this proposition is that each particle in the growth-fragmentation process can either be coloured red, if its genealogy goes extinct, or blue, if its genealogy survives forever. In the case where the initial particle is blue, the branching operator $H$ describes the colours of the offspring when a fragmentation occurs. In particular, the production of blue particles is described by the function $f$, whereas $g$ describes the production of red particles. We can see from \eqref{jointfrag} that the first summand in the expectation on the righthand side of~\eqref{jointfrag} describes the case where a particle fragments and both fragments are blue. The second (resp. third) summand then describes the case where the fragment of size $xV$ (resp. $x(1-V)$) is blue and the other is red.

The proof of this proposition follows from \cite[Section 2.3]{HHK-snte} by setting $\varsigma(x) = B + \kr$ and using the specific form of the branching generator given in~\eqref{NLGFop}. In particular, the reader may find the proof for the exact formulation of $H$ in the proof of \cite[Proposition]{HHK-snte}.

In order to prove the strong law, we will first prove it for
the blue process $(\Zb,\PP^B)$, and then show that this implies
the same result for $(\Zb,\PP)$.

The combination of \autoref{l:colouring} and \autoref{p:skele} gives
the following identity, which will be very useful.

\begin{align}
  \EE^B_x[\langle f,\Zb(t)\rangle] 
  = \EE_x\left[\sum_{u}f(Z_u(t))\Indic{C_u = b} \, \bigg| \, C_\emptyset = b \right] \notag
  &= \frac{1}{p}\EE_x\left[\sum_{u}f(Z_u(t))\Indic{C_u = b} \Indic{C_\emptyset = b} \right] \notag \\
  &= \frac{1}{p}\EE_x\left[\sum_{u} f(Z_u(t))\Indic{C_u = b} \right] \notag \\
  &= \frac{1}{p}\EE_x[\langle f, p\Zb(t) \rangle] \notag \\
  &= \EE_x[\langle f, \Zb(t) \rangle], \label{bluesame}
\end{align}
where the first equality follows from the proposition;
the third equality uses the fact that the sum is empty if the initial particle is not blue;
and the penultimate equality comes from the lemma.

An immediate consequence is that
if $(\lambda_*, 1, \nu)$ is the eigen-triple 
for $(\Zb, \mathbb{P})$, then $(\lambda_*, 1/p, p\nu)$ is the 
eigen-triple for the blue process $(\Zb, \mathbb{P}^B)$. 
Moreover, the process $Y$ defined by
$\EE^B_x[f(Y_t)] = e^{-\lambda_* t}\EE_x^B[\langle f,\Zb(t)\rangle]$
is the same in distribution as the process $Y$ we defined earlier
for $(\Zb,\PP)$.

Under the measures $\PP^B$, we retain our notation
$N_t$ for the number of blue particles 
alive at time $t \ge 0$. Since $\lambda$ is still the leading eigenvalue for the blue process,
it follows that 
\[
  M_t \coloneqq e^{-\lambda t}N_t
\]
is a positive martingale under $\mathbb{P}_x^B$, for 
each $x \in (0, c]$. Let $M_\infty$ denote its limit.
We will show in what follows that $M$ is $L^2(\PP_x^B)$-convergent. 

\smallskip

Our intermediate strong law is as follows.

\begin{theorem}\label{t:blue-slln}
  For all $x \in (0,\infty)$ and continuous $f$ with $\lVert f \rVert <1$,
  \[
    e^{-\lambda t}
    \langle f , \Zb(t) \rangle
    \to \langle f,\nu\rangle M_\infty,
  \]
  $\PP^B_x$-almost surely and in $L^1(\PP^B_x)$.
\end{theorem}

\subsection{Proof of \autoref{t:blue-slln}}

In this part, we follow closely the ideas of \cite{EHK-slln},
though of course our processes are very different in nature.

\begin{lemma}\label{l:blue-Ut}
  Let $U_t = e^{-\lambda_* t} \langle f, \Zb(t) \rangle$. Then, for any increasing sequence $(m_n)_{n \ge 1}$,
  \[ U_{(m_n+n)\delta} - \EE_x^B[ U_{(m_n+n)\delta} \mid \FF_{n\delta}] \to 0, \]
  as $n\to\infty$, $\PP^B_x$-a.s.\ and in $L^1(\PP^B_x)$.
\end{lemma}

\begin{proof}
  For convenience, we will write $m$ instead of $m_n$.
  For almost sure convergence, by the Borel-Cantelli lemma, it is sufficient to show that, for all $\epsilon>0$,
  \[
    \sum_{n\ge 1} \PP^B_x\Bigl( \bigl\lvert U_{(m+n)\delta}
    -\EE^B_x[U_{(m+n)\delta} \mid \FF_{n\delta}] \bigr\rvert >\epsilon\Bigr) < \infty.
  \]

  We first note that 
  \begin{align*}
    \EE^B_x\left[ \bigg| U_{(m+n)\delta} - \EE^B_x[ U_{(m+n)\delta} \mid \FF_{n\delta}]\bigg|^2 \bigg| \FF_{n\delta}\right]
    &=
    \EE^B_x\Biggl[ 
      \bigg|
      \sum_u e^{-n\delta\lambda} \bigl(U_{m\delta}^{(u)} - \EE^B_{Z_u(n\delta)}[U_{m\delta}^{(u)}]\bigr)
      \bigg|^2 
    \mathbin{\Bigg\vert}\FF_{n\delta} \Biggr],
  \end{align*}
  where, conditional on $\FF_{n\delta}$, the $U_{m\delta}^{(u)}$ are independent and distributed as $(U, \PP_{Z_u(n\delta)})$ 
  for $u \in \Uu_{n\delta}$.
  It follows that
  \begin{align}
    \EE^B_x\left[ \bigg| U_{(m+n)\delta} - \EE^B_x[ U_{(m+n)\delta} \mid \FF_{n\delta}]\bigg|^2 \bigg| \FF_{n\delta}\right]
    &=
    \EE^B_x\Biggl[ 
      \bigg|
      \sum_u e^{-n\delta\lambda} \bigl(U_{m\delta}^{(u)} - \EE^B_{Z_u(n\delta)}[U_{m\delta}^{(u)}]\bigr)
      \bigg|^2 
    \mathbin{\Bigg\vert}\FF_{n\delta} \Biggr] \notag\\
    &=
    \EE^B_x\biggl[
      \sum_u e^{-2n\delta\lambda} \big|U_{m\delta}^{(u)} - \EE^B_{Z_u(n\delta)}[U_{m\delta}^{(u)}]\big|^2
    \mathbin{\bigg\vert} \FF_{n\delta} \biggr] \notag
    \\
    &\le 2\EE^B_x\left[ \sum_u e^{-2n\delta\lambda}\left(\big| U_{m\delta}^{(u)} \big|^2 + \big|\EE^B_{Z_u(n\delta)}[U_{m\delta}^{(u)}]\big|^2\right)
    \bigg| \FF_{n\delta}\right] \notag
    \\
    &\le 4 e^{-2\lambda n\delta}\sum_u \EE^B_{Z_u(n\delta)}\left[\left(U_{m\delta}^{(u)}\right)^2\right] \notag
    \\
    &\le 4e^{-n\delta\lambda} \lVert f \rVert^2 \sup_{x,t} \EE^B_x[(M_t)^2] \cdot M_{n\delta}. \label{UB1}
  \end{align}
  where, for the second equality, we used that the sum is over conditionally independent, zero-mean
  summands, for the first inequality we have used that $|a + b|^2 \le 2(|a|^2 + |b|^2)$ and for the second inequality we have used 
  Jensen's inequality.

  A similar calculation to that of \eqref{bluesame} shows that 
  $\EE^B_x[(M_t)^2] = p\EE_x[M_t^2]$, 
  and hence, $M$ is $L^2(P_x^B)$-bounded, uniformly in $x \in (0,c]$. 
  From Doob's inequality, it follows that
  $M$ is $L^2(\PP_x^B)$-convergent and hence $L^1(\PP^B_x)$-convergent.
  Therefore,
  \[ 
    \EE^B_x\Bigl[ \bigl( U_{(m+n)\delta} - \EE^B_x[ U_{(m+n)\delta} \mid \FF_{n\delta}]\bigr)^2\Bigr]
    \le 2\lVert f\rVert^2 \sup_{x,t} \EE^B_x[M_t^2] \cdot e^{-n\delta\lambda},
  \]
  which is summable, so the Markov inequality completes the proof of the almost sure convergence.
\end{proof}

\begin{lemma} \label{l:blue-slln-lattice}
  \autoref{t:blue-slln} holds when restricted to lattice times:
  \[
    \lim_{n \to \infty}e^{-n\delta\lambda} \langle f, \Zb(n\delta\lambda)\rangle = \langle f, \nu \rangle M_\infty,
  \]
  $\PP^B_x$-almost surely and in $L^1(\PP^B_x)$.
\end{lemma}

\begin{proof}
  Noting that 
  \[
    \EE^B_x\bigl[ 
      e^{-\lambda 2n\delta } \langle f, \Zb(2n\delta) \rangle \mathbin{\big\vert} \FF_{n\delta} 
    \bigr]
    = e^{-\lambda n\delta} \sum_u e^{-\lambda n\delta}\EE^B_{Z_u(n\delta)}[\langle f, \Zb(n\delta) \rangle],
  \]
  we have
  \begin{align*}
    e^{-\lambda 2n\delta}\langle f, \Zb(2n\delta) \rangle
    &=
    e^{-\lambda 2n\delta}\langle f, \Zb(2n\delta) \rangle 
    - e^{-\lambda 2n\delta}\EE^B_x\bigl[ \langle f, \Zb(2n\delta) \rangle 
    \mathbin{\big\vert} \FF_{n\delta} \bigr] \\
    & \quad {}
    + e^{-\lambda n\delta}
    \sum_u
    \Bigl( e^{-\lambda n\delta}\EE^B_{Z_u(n\delta)}\bigl[ \langle f, \Zb(n\delta)\rangle\bigr] 
    - \langle f, \nu\rangle \Bigr) \\
    & \quad {} + \langle f,\nu\rangle M_{n\delta}.
  \end{align*}
  The first term goes to zero a.s.\ and in $L^1(\PP^B_x)$, by setting $m_n = n$ in the preceding lemma.
  The last term approaches $\langle f,\nu \rangle M_\infty$ by the martingale convergence
  discussed previously.

  If we define
  \[
    R_{s,t}
    = \sum_u e^{-\lambda t} 
    \lvert e^{-\lambda s} \EE^B_{Z_u(t)}[ \langle f,\Zb(s) \rangle ] - \langle f,\nu\rangle \rvert,
  \]
  then, using again \eqref{bluesame}, we have
  \[
    R_{s,t} = \sum_u e^{-\lambda t} 
    \lvert e^{-\lambda s} \EE^B_{Z_u(t)}[ \langle f,\Zb(s) \rangle ] - \langle f,\nu\rangle \rvert
    = \sum_u e^{-\lambda t} 
    \lvert  \EE^B_{Z_u(t)}[ f(Y_s) ] - \langle f,\nu\rangle \rvert
  \]
  it is sufficient to show that
  $\lim_{n\to\infty} R_{n\delta, n\delta} = 0$ almost surely.
  We now wish to apply \autoref{p:asym}. This was proven under the measures $\PP_\cdot$, but
  as remarked above, the distribution of the spine $Y$ is the same under $\PP^B_\cdot$.
  By \autoref{p:asym}, choose $w \in (0,r)$ such that $\psi_\eta(-w)<0$;
  then,
  there exist $k>0$ and $C>0$
  such that for all $x\in (0,c]$ and $t\ge 0$,
  \[
    \bigl\lvert \EE^B_x[ f(Y_s)] - \langle f,\nu\rangle \bigr\rvert
    \le 
    \lVert f_w \rVert \bigl( (c/x)^{w} + C \bigr) e^{-ks},
  \]
  where we recall $f_w(x) = (x/c)^w f(x)$.

  Applying this in our situation, we have
  \[
    \EE_x[R_{s,t}]
    \le 
    e^{-ks}
    \lVert  f_w \rVert
    \EE_x\bigl[ (c/Y_t)^{w} + C \bigr]
    =
    e^{-ks}
    \lVert f_w \rVert
    \EE_x[ e^{w(b-\eta^b_t)} + C ]
    .
  \]

  We estimate:
  \begin{align*}
    \EE_x[e^{w(b-\eta^b_t)}]
    = \EE_x[e^{w(\bar\eta_t\vee b - \eta_t)}] 
    &\le \EE_x[e^{w(b+\bar\eta_t-\eta_t)}] \\
    &= \EE[e^{w(b+\bar\eta_t-\eta_t)}] \\
    &= e^{wb} \EE[e^{-w\underline{\eta}_t}] \\
    &\le e^{wb} \EE[e^{-w\underline{\eta}_\infty}].
  \end{align*}
  We will show that the right-hand side is finite, using
  the Wiener-Hopf factorisation of $\eta$.

  Let $\hat{H}$ denote the descending ladder height process of $\eta$
  and $\hat{\kappa}(0,\cdot)$ its Laplace exponent.
  By \cite[p.~178]{Kyp2}, we know that
  $\underline{\eta}_\infty \overset{d}{=} -\hat{\mathbb{H}}_{\mathbf{e}_\chi}$,
  where $\hat{\mathbb{H}}$ is an unkilled version of $\hat{H}$, and
  $\mathbf{e}_{\chi}$ is an independent exponential random variable with
  rate $\chi = \hat{\kappa}(0,0)>0$.

  By decomposing the expectation according to the law of $\mathbf{e}_{\chi}$,
  we see that $\EE[e^{w \hat{\mathbb{H}}_{\mathbf{e}_\chi}}] < \infty$
  if and only if $\EE[e^{w\hat{H}_1}]<\infty$ and $\hat{\kappa}(0,-w) > 0$.
  Furthermore, by \cite[Theorem 7.8]{Kyp2}, we know that the Lévy measures of
  $\hat{H}$ and $\eta$ are related by $\Pi_{\hat{H}}(\dd y) = \Pi_\eta(y,\infty)\,\dd y$,
  so if 
  \begin{equation}\label{e:eta1}
    \EE[e^{w \eta_1}] < \infty,
  \end{equation}
  then $\EE[e^{w\hat{H}_1}] < \infty$ also.

  Now, by assumption, our choice of $w$
  satisfies \eqref{e:eta1} and $\psi_\eta(-w)<0$, so the results of \cite[\S 6.5.2]{Kyp2} imply
  that $\hat{\kappa}(0,-w) = \frac{\psi_\eta(-w)}{-w} > 0$. We conclude that
  $\EE[e^{-w\underline{\eta}_\infty}]<\infty$.

  It follows that
  \[ \EE_x[R_{s,t}] \le e^{-ks} \lVert f_w \rVert C' , \]
  with $C' = e^{wb} \EE[e^{-w\underline{\eta}_\infty}] + C$.
  Hence, we have convergence in $L^1$. Finally, by the Markov property,
  \[
    \sum_{n\ge 1} \PP_x(R_{n\delta,n\delta} > \epsilon) < \infty.
  \]
  The Borel-Cantelli lemma implies that $\lim_{n\to\infty} R_{n\delta,n\delta} = 0$
  almost surely.
\end{proof}

\begin{proof}[Proof of \autoref{t:blue-slln}]
  To complete the proof of \autoref{t:blue-slln}, we need to pass from lattice time to continuous time.
  For the $L^1(\PP^B_x)$ convergence, the Croft-Kingman lemma yields the result. For the almost sure part,
  the proof is identical to the proof given in \cite{EHK-slln}, and so we omit it.
\end{proof}

\subsection{Proof of \autoref{t:slln}}

In this section, we work under the probability measure $\PP$, and will make use
of both the full growth-fragmentation process $\Zb$ and the part of it
coloured blue, which we denote by $\Zb^B$ (and which, if it exists,
has the law of $(\Zb,\PP^B)$). The intrinsic martingale
of the blue tree will be denoted $M^B$.

First note that the $L^1$ convergence for the full process follows easily from the relation \eqref{bluesame}.
For the almost sure convergence, the proof follows the same idea as \cite[\S 4]{HHK-snte}, using the following proposition,
a discussion of whose proof can also be found in that reference.
\begin{proposition}\label{p:trick}	
  Let $(\Omega, \mathcal{F}, (\mathcal{F}_t,t\geq 0), \mathbb{P})$ be a filtered probability space and define $\mathcal{F}_\infty 
  \coloneqq \sigma(\cup_{i = 1}^\infty\mathcal{F}_t)$. Suppose $(U_t, t\geq 0)$ is an $\mathcal{F}$-measurable non-negative 
  process such that $\textstyle{\sup_{t\geq 0}U_t}$ has finite expectation and  $(\mathbb{E}(U_t | \mathcal{F}_t), t\geq0)$ is 
  c\`adl\`ag. If 
  \[
    \lim_{t\to\infty}\mathbb{E}(U_t | \mathcal{F}_\infty) = Y, \text{ a.s,}
  \]
  then
  \[
    \lim_{t\to\infty}\mathbb{E}(U_t | \mathcal{F}_t) =Y, \text{ a.s.}.
  \]
\end{proposition}

\bigskip

To put the above proposition in the context of the growth-fragmentation setting, set $U_t = e^{-\lambda t}\langle f, 
\Zb^B(t)\rangle$, for $f$ satisfying the conditions of ~\autoref{t:slln}, and recall that  $(\FF_t, t \ge 0)$ is the filtration generated by 
the growth-fragmentation process $(\Zb(t), t \ge 0)$. Note that we can easily bound $(U_t,t\geq 0)$ by a multiple of 
$(M^B_t,t\geq 0)$ and hence we automatically get that $\textstyle{\sup_{t\geq 0}U_t}$ has a second, and hence first, moments 
thanks to \eqref{bluesame}.
Due to \autoref{t:blue-slln} and the fact that $\Zb^B(t)$ is $\FF_\infty$-measurable, $U_t = \EE(U_t | 
\FF_\infty) $ and hence
\[
  \lim_{t\to\infty}\EE(U_t | \FF_\infty) = \langle f, \nu \rangle M_\infty^B,
\]
$\PP_{x}$-almost surely, for $x\in(0,c]$.

\smallskip

Using \eqref{bluesame}, we get
\begin{align*}
  \EE(U_t | \FF_t) &= \EE(e^{-\lambda t}\langle f, \Zb^B(t)\rangle| \FF_t)
  = e^{-\lambda t}p\langle f, \Zb(t)\rangle.
\end{align*}

Combining this with \autoref{p:trick} yields
\begin{equation}
  \lim_{t\to\infty}e^{-\lambda t}\langle f, X_t\rangle = \langle f, \nu\rangle M^B_\infty/p ,
  \label{gp}
\end{equation}
$\PP_x$-almost surely.

\smallskip

To complete the proof of almost sure convergence,
we need to show that $M_\infty^B/p = M_\infty$, almost surely.
To do so, take $f=1$ in \eqref{gp} and observe that the left-hand side is $M_\infty$.
\qed

\section{Long-term behaviour in the transient regime}
\label{s:transient}

Finally, in this last section, we will prove \autoref{t:transientlln}. 
Consider the case where $\Zb$ is supercritical but \eqref{T} holds, so that
$\lambda = q_* = \lambda_* + \inf\psi_\eta < \lambda_*$.

To simplify the exposition, in this section we will assume that $\mathtt{k} = 0$,
that is, that there is no killing of cells. In the general case, the results
of this section can be proved using the skeleton decomposition, as
in the previous section.

Now, since we are assuming that \eqref{T} holds, we no longer have the advantage of
being able to use $M_t$ to define a change of measure, as in the previous section.
However, in this case, we are able to use the function
$\ell(x) = L_{x,c}(\lambda) = (x/c)^{\arginf \psi_\eta}$
to yield a useful \textit{supermartingale} change of measure. 

Our techniques in this section are inspired by a combination of \cite{BW-lln}
and \cite{EHK-slln}.

\begin{proposition}\label{p:YS}
  \begin{enumerate}
    \item
      The process
      $S_t = e^{-(\lambda - \lambda_*)t} \frac{\ell(Y_t)}{\ell(x)}
      = e^{-\inf\psi_\eta \cdot t} \frac{\ell(Y_t)}{\ell(x)}$
      is a $\PP_x$-supermartingale for the natural filtration $\FF^Y$
      of $Y$. Under the change
      of measure
      \[
        \left.\frac{\dd \tilde{\PP}_x}{\dd \PP_x}\right\rvert_{\FF^Y_t} 
        = e^{-\inf \psi_\eta \cdot t} \frac{\ell(Y_t)}{\ell(x)},
      \]
      $Y$ is the exponential of a Lévy process with
      Laplace exponent $\tilde\psi_\eta(q) = \psi_\eta(q+\arginf\psi_\eta) - \inf\psi_\eta$,
      reflected in the level $c$ and killed according to the
      multiplicative functional
      $t \mapsto e^{-a\cdot \arginf\psi_\eta \cdot \int_0^t \Indic{Y_s = c}\,\dd s}$.

    \item
      The process $\mathcal{S}_t = e^{-\lambda t}\frac{1}{\ell(x)}\langle \ell, \Zb\rangle$
      is a $\PP_x$-supermartingale for the natural filtration of $\Zb$.
      A measure $\tilde{\PP}_x$ supporting $\Zb$ and the additional
      random variable $\zeta$ can be defined as follows:
      \[
        \tilde{\PP}_x( F_t \Indic{\zeta > t} ) = \PP_x(F_t \mathcal{S}_t), \qquad F_t \in \FF_t.
      \]
      Under this measure,
      the process $\Zb$ has the following decomposition. There is a
      single distinguished cell whose mass
      has the distribution of the process $Y$ under $\tilde{\PP}_x$,
      and at every jump $\Delta Y_t$ of this cell, a
      copy of $\Zb$ under measure $\PP_{-\Delta Y_t}$ is
      introduced, which we can denote $\Zb^{[t]}$. That is,
      $\Zb$ has the same distribution as the process
      \[ t\mapsto \delta_{Y_t} + \sum_{0<s \le t} \Zb^{[s]}(t-s). \]
      The random variable $\zeta$ is the lifetime of $Y$.

    \item
      Assume that $\lambda > 0$. Then,
      $\sup_{x \in (0,c]} \sup_{t\ge 0} \ell(x) \EE_x[\mathcal{S}_t^2] < \infty$.
  \end{enumerate}
\end{proposition}
\begin{proof}
  \begin{enumerate}
    \item
      Let $\hat{\ell}(y) = e^{\arginf \psi_\eta \cdot(y-b)}$, so
      that $\ell(Y_t) = \hat{\ell}(\eta^b_t)$.

      The $\PP_x$-supermartingale property of $S$ is equivalent to the
      statement that
      \[ \EE_y[e^{-\inf\psi_\eta \cdot t}
      \hat{\ell}(\eta^b_t)/\hat{\ell}(x)] \le 1, \]
      since $\eta^b$ is Markov (where $y=\log x$.)
      Recall that
      $\EE_y[e^{\alpha\eta_t-\psi_\eta(\alpha)t-\alpha y}] = 1$, for any $\alpha \ge 0$. Then,
      \begin{align*}
        \EE_y[e^{\alpha \eta^b_t - \psi(\alpha)t - \alpha y}]
        &= e^{-\alpha y - \psi_\eta(\alpha)t} \EE_y[e^{\alpha (\eta_t - (\bar{\eta}_t -b)\vee 0)}] \\
        &= \EE_y[e^{-\alpha(\bar{\eta}_t - b)\vee 0}] \le 1.
      \end{align*}
      Setting $\alpha = \arginf\psi_\eta > 0$ completes
      the proof of the supermartingale property.

      We turn to the characterisation of $Y$. It appears to be simplest to
      demonstrate this using evolution equations. Therefore, let
      $(\Phi_t)_{t\ge 0}$ represent the semigroup of $Y$ under $\PP_\cdot$,
      and $(\tilde{\Phi}_t)_{t\ge 0}$ the same object under $\tilde{\PP}_\cdot$.

      It is simple to show, using the same ideas as in section \ref{s:sg},
      that $\Phi$ satisfies the evolution equation
      \begin{align*}
        \Phi_t f(x)
        &= f(xe^{a \cdot t\wedge T_c(x)}) \\
        & \quad {}
        + \int_0^t \dd s \int_0^1
        (\Pi_\eta \circ \exp^{-1})(\dd v)\,
        \bigl[ \Phi_{t-s} f(xe^{a \cdot s\wedge T_c(x)} v)
          - \Phi_{t-s} f(xe^{a \cdot s \wedge T_c(x)})
        \bigr].
      \end{align*}
      The semigroups $\Phi$ and $\tilde{\Phi}$ are related by the formula
      \[ \tilde{\Phi}_t f(x) 
        = e^{-\inf \psi_\eta \cdot t}
      \frac{1}{\ell(x)} \Phi_t(f\ell)(x) . \]
      Rewriting the above evolution equation in terms of $\tilde{\Phi}$,
      performing some algebraic manipulation and making use
      of Dynkin's integral identity (\autoref{l:dynkin})
      we obtain the following evolution equation for $\tilde{\Phi}$:
      \begin{align*}
        \tilde{\Phi}_tf(x)
        &=
        f(xe^{a \cdot t \wedge T_c(x)})
        - \int_{T_c(x)}^t \alpha a
        \tilde{\Phi}_{t-s}f(xe^{a \cdot s\wedge T_c(x)})\, \dd s
        \\
        & \quad {}
        + \int_0^t \dd s
        \int_0^1
        (\tilde{\Pi}_\eta\circ \exp^{-1})(\dd v) \,
        \bigl[
          \tilde{\Phi}_{t-s} f(xe^{a \cdot s \wedge T_c(x)}v)
          - \tilde{\Phi}_{t-s} f(xe^{a \cdot s \wedge T_c(x)})
        \bigr],
      \end{align*}
      where $\alpha = \arginf \psi_\eta$ and $\tilde{\Pi}_\eta$ is the
      Lévy measure of the Lévy process with Laplace
      exponent $\tilde{\psi}_\eta$ given in the statement.
      It is clear that this is the evolution equation associated
      with the process described in the statement of the result.

      It remains to show that the above evolution equation characterises
      $\tilde{\Phi}$. Since the killing and jump rates are bounded,
      this follows using Gronwall's inequality exactly as in
      section \ref{s:sg}.

    \item
      This proof is very similar to the classical methods of \cite{HH-spine}, and
      we give only an outline.
      The first step is to introduce an additional (killed) process $I_t$
      on $\Uu$ by
      defining
      \[
        \tilde{\PP}_x( F_t \Indic{I_t = u} ) 
        = e^{-\lambda t}\frac{1}{\ell(x)}\PP_x(F_t \ell(Z_u(t))), \qquad F_t \in \FF_t,
      \]
      and declaring $\zeta$ to be the killing time of $I$.
      This random variable $I_t$ indicates the index of the distinguished `spine' cell.

      It follows that
      \[
        \tilde{\EE}_x[ f(Z_{I_s}(s), s \le t) ]
        = e^{-\lambda t}\frac{1}{\ell(x)}
        \sum_{u} \EE_x[ f(Z_u(s), s\le t) \ell(t) ],
      \]
      and in particular, taking part (i) into account, we see that under
      $\tilde{\PP}_x$, $t \mapsto Z_{I_t}(t)$ has the same distribution as
      the process $Y$.

      Denote by $T_k(u)$ the $k$-th jump of the cell labelled $u$ (or its
      ancestors), and by $T_k$ the $k$-th jump of $Y$.
      Since all processes involved
      are Markov, it suffices to check the decomposition at fixed times,
      and we will focus first on the case where $T_1(I_t) \le t < T_2(I_t)$;
      that is, between the first and second jumps of the spine.
      We define an operator $\rr$ on $\Uu$ which removes the prefix,
      i.e., if $u = u_1u_2u_3\dotsb$, $\rr u = u_2u_3\dotsb$.
      Now let $f$ and $g$ be measurable functions.
      \begin{align*}
        &
        \tilde{\EE}_x
        \left[
          f(Z_{I_t}(t)) g(Z_v(t), v \ne I_t)
          \Indic{\zeta > t}
          \Indic{T_1(I_t) \le t < T_2(I_t)}
        \right]
        \\
        &\quad {} =
        \sum_u
        e^{-\lambda t}
        \frac{1}{\ell(x)}
        \EE_x
        \left[
          f(Z_u(t)) g(Z_v(t), v\ne u)
          \ell(Z_u(t))
          \Indic{T_1(u) \le t < T_2(u)}
        \right]
        \\
        &\quad {} =
        \sum_u e^{-\lambda t} \frac{1}{\ell(x)}
        \EE_x\left[
          \EE_{Z_u(T_1(u))}[f(Z_{\rr u}(t-s)\ell(Z_{\rr u}(t-s))\Indic{t-s<T_1(\rr u)}]
        \right. \\
        & \qquad\qquad\qquad {} \times
        \left.
          \EE_{-\Delta Z_u(T_1(u))}[ g(Z_{\rr v}(t-s), v \ne u) ]
          \rvert_{s=T_1(u)}
          \Indic{T_1(u) \le t}
        \right]
        \\
        & \quad {}
        =
        \tilde{\EE}_x\left[
          \tilde{\EE}_{Y_{T_1}}[ f(Y_{t-s}) \Indic{t-s<T_1}]
          \EE_{-\Delta Y_{T_1}}[ g(\Zb(t-s))]
          \rvert_{s=T_1}
          \Indic{T_1 \le t}
        \right]
        \\
        & \quad {}
        =
        \tilde{\EE}_x\left[ f(Y_{t}) \EE_{-\Delta Y_{T_1}}[ g(\Zb(t-s))]\rvert_{s=T_1}
          \Indic{T_1 \le t < T_2}
        \right].
      \end{align*}
      This proves the claim on the event $\{ T_1(I_t)\le t < T_2(I_t)\}$, and the full
      proof proceeds by induction, in each case conditioning at time $T_1(I_t)$.

    \item
      This problem can be approached using the spine decomposition.
      We have that
      \[
        \ell(x)\EE_x[\mathcal{S}_t^2] = \ell(x)\tilde{\EE}_x[\mathcal{S}_t]
        = e^{-\lambda t} \tilde{\EE}_x[\ell(Y_t)] 
        + \tilde{\EE}_x\left[ 
          \sum_{s<t} e^{-\lambda s} \cdot e^{-\lambda(t-s)} 
          \langle \ell, \mathbf{Z}^{[s]}(t-s) \rangle 
        \right ],
      \]
      where the sum is over the jump times $s$ of $Y$ up until its lifetime.
      Since $\ell$ is bounded, the first term is bounded in $t,x$.
      To bound the second term, it is helpful to recall that $Y$ is the
      exponential of a Lévy process $\eta$ reflected in $b$,
      with extra killing.
      Therefore, if we write $\FF^Y$ for the filtration of the
      distinguished cell whose evolution has path $Y$, we have
      \begin{align*}
        \tilde{\EE}_x \left[ \sum_{s < t} e^{-\lambda s} e^{-\lambda(t-s)} 
          \langle \ell,\Zb^{[s]}(t-s) \rangle 
        \middle| \FF^Y_\infty \right]
        &= \sum_{s<t} e^{-\lambda s} \ell(-\Delta Y_s) \EE_{-\Delta Y_s}[\mathcal{S}_{t-s}] \\
        &\le \sum_{s<t} e^{-\lambda s} \ell(-\Delta Y_s).
      \end{align*}
      Using the tower property and the
      predictable compensator of the jumps of $Y$, we obtain
      \begin{align*}
        \tilde{\EE}_x \left[ \sum_{s < t} e^{-\lambda s}
        e^{-\lambda(t-s)} \langle \ell,\Zb^{[s]}(t-s) \rangle \right]
        &\le \tilde{\EE}_x\left[ \int_0^t e^{-\lambda s} 
        \int_{-\infty}^0 \ell(Y_{s-}(1-e^z)) \, \tilde{\Pi}_\eta(\dd z) \right],
      \end{align*}
      Again, since $\ell$ is bounded and $\tilde{\Pi}_\eta$ is finite
      (because $\arginf\psi_\eta > 0$),
      this term is also bounded in $t,x$.
  \end{enumerate}
\end{proof}

\begin{remark}
  \begin{enumerate}
    \item
      Though the process $Y$ under $\tilde{\PP}_x$ is a little unusual,
      it can be shown that its lifetime $\zeta$ satisfies
      \begin{align*}
        \int_0^\infty e^{-qt} \tilde{\PP}_c(\zeta >t) \, \dd t
        &= \int_0^\infty e^{-qt} 
        \EE[e^{-\arginf\psi_\eta \cdot (\bar{\eta}_t-\eta_t) - \inf\psi_\eta \cdot t}]\,\dd t \\
        &= \frac{1}{\Phi(q+\inf\psi_\eta)} \frac{q}{\Phi(q+\inf\psi_\eta)-\arginf\psi_\eta},
      \end{align*}
      This can be proved using the Wiener-Hopf factorisation, in particular the Laplace
      transform of the downward ladder height subordinator for a spectrally negative Lévy process
      \cite[Theorem 6.15(ii) and equation (6.35)]{Kyp2}.
    \item
      The preceding proposition makes sense even in the presence of Gaussian fluctuations
      or infinite jump activity. The process $Y$ is killed according to the local
      time at $c$. In the case of diffusions this is known as `elastic' boundary behaviour
      \cite[Example IV-5.5]{IW-sdes}.
    \item
      Part (i) of the previous result proves that $\hat\ell$ is the extremal
      excessive function of the reflected Lévy process $\eta_b$ associated
      with the Martin boundary point $b$. This may be of interest
      as an example in the context of spectrally negative Markov processes.
  \end{enumerate}
\end{remark}

\begin{lemma}
  Let $f$ be such that $\lVert f/\ell \rVert_\infty < \infty$, and fix $m\ge 0$. Define
  $U_t = e^{-\lambda t} \langle f,\Zb(t)\rangle$. Then,
  \[ U_{(m+n)\delta} - \EE_x[U_{(m+n)\delta} \mid \FF_{n\delta} ] \to 0, \]
  as $n\to\infty$, $\PP_x$-a.s.\ and in $L^1(\PP_x)$.
\end{lemma}
\begin{proof}
  To be concise, let $s=m\delta$ and $t=n\delta$.
  Using the conditional independence of the zero-mean summands, we get
  \begin{align*}
    \EE_x\left[ \left( U_{s+t} - \EE_x[U_{s+t} \mid \FF_t] \right)^2 \middle| \FF_t \right]
    &= \EE_x\left[ e^{-2\lambda t}
      \sum_i \left( e^{-\lambda s} \langle f, \Zb^{(i)}(s)\rangle - \EE_{Z_i(t)}[e^{-\lambda s} \langle f,\Zb(s)\rangle] \right)^2
    \middle| \FF_t \right] \\
    &= \sum_{i} e^{-2\lambda t} \left(
      \EE_{Z_i(t)}[(e^{-\lambda s} \langle f,\Zb(s)\rangle)^2]
      - \EE_{Z_i(t)}[e^{-\lambda s} \langle f,\Zb(s)\rangle]^2
    \right) \\
    &\le 2\lVert f/\ell\rVert_\infty^2 \sum_{i} e^{-2\lambda t} \EE_{Z_i(t)}[ (e^{-\lambda s} \langle \ell, \Zb(s)\rangle)^2]
    \\
    &\le 2\lVert f/\ell\rVert_\infty^2 \sum_{i} e^{-2\lambda t} \ell(Z_i(t))^2 \EE_{Z_i(t)}[\mathcal{S}_s^2]
    \\
    &\le 2\lVert f/\ell\rVert_\infty^2 
    \cdot \sup_{y,u} \ell(y)\EE_y[\mathcal{S}_u^2]
    \cdot e^{-\lambda t} \ell(x) \mathcal{S}_t.
  \end{align*}
  Hence, using that $\mathcal{S}$ is a supermartingale,
  \begin{align*}
    \EE_x\left[ \left( U_{s+t} - \EE_x[U_{s+t} \mid \FF_t] \right)^2 \right]
    &\le \text{const} \cdot e^{-\lambda t} \ell(x). 
  \end{align*}
  Hence, the $L^1(\PP_x)$-convergence to zero holds. Moreover,
  this is summable in $n$ (recall that $t=n\delta$) and so,
  using an application of Markov's inequality and the Borel-Cantelli lemma,
  the almost sure convergence to zero holds too.
\end{proof}

We are now able to state and prove the following rephrasing of our second main theorem.
To understand the connection with the statement in the introduction,
which is in terms of the `cumulant' $\kappa$, observe that
$\kappa(0) = \lambda_*$ and that
\[ \psi_\eta(q) = \kappa(q) - \kappa(0) .\]
Hence, $\inf\kappa = \inf\psi_\eta + \lambda_*$
and $\arginf\kappa = \arginf\psi_\eta$, so that in particular
$\lambda = \inf\kappa$.

\begin{theorem}[Rephrasing of \autoref{t:transientlln}]
  Assume that $\lambda > 0$ 
  and that $f \from (0,c]\to\RR$ is continuous and bounded with
  $f(x) = O(x^{\arginf\psi_\eta})$ as $x\to 0$. Then,
  \[ e^{-\lambda t} \langle f,\Zb(t)\rangle \to 0, \]
  $\PP_x$-almost surely and in $L^1(\PP_x)$.
\end{theorem}

\begin{proof}
  We begin with the proof for lattice times,
  and for simplicity, we assume $f\ge 0$.
  For $\delta > 0$, we have,
  as in the recurrent case,
  \begin{align*}
    e^{-\lambda (s+t)} \langle f, \Zb(s+t) \rangle
    &=
    e^{-\lambda (s+t)} \langle f, \Zb(s+t) \rangle
    - \EE_x\bigl[ e^{-\lambda (s+t)} \langle f, \Zb(s+t) \rangle
    \mathbin{\big|} \FF_{t} \bigr] \\
    & \quad {}
    + e^{-\lambda t}
    \sum_i
    \EE_{Z_i(t)}\bigl[e^{-\lambda s} \langle f, \Zb(s)\rangle\bigr].
  \end{align*}
  The first term converges to zero in the sense stated along lattice times, by the preceding lemma.
  The second term can be expressed
  \begin{align*}
    \sum_i
    e^{-\lambda t}
    \EE_{Z_i(t)}\bigl[e^{-\lambda s} \langle f, \Zb(s)\rangle\bigr]
    &= \sum_{i} e^{-\lambda t} \ell(Z_i(t))
    \tilde{\EE}_{Z_i(t)}\left[\frac{f(Y_s)}{\ell(Y_s)}\right]
    \le \lVert f/\ell\rVert_\infty \mathcal{S}_t.
  \end{align*}
  Since $\mathcal{S}$ is a positive supermartingale, it converges almost
  surely \cite[Corollary II.2.11]{RY-cmbm}.
  Moreover, $\EE_x[\mathcal{S}_t] = \tilde{\PP}_x(\zeta>t) \to 0$, where $\zeta$ is
  the lifetime of $Y$. Hence, $\mathcal{S}_t$ converges to zero
  in $L^1(\PP_x)$. It follows that the almost sure limit of $\mathcal{S}_t$
  is also zero.

  The extension of this limit from lattice to continuous times follows the same
  idea as in the recurrent regime.
\end{proof}

\begin{remark}
  We believe that $\lambda$ is the best (i.e., smallest)
  exponential rate that can be obtained for the
  above theorem. However, it may be possible to obtain more precise results
  by considering the functional
  $\hat{L}_{x,y}(p) = \EE_x[e^{-(\lambda - \lambda_*)H(y)} H(y)^{p-1}; H(y)<\infty]$
  instead of $L$ used in \autoref{p:asym}.
  We leave this question for future work.
\end{remark}

\appendix

\section{An inhomogeneous version of Dynkin's integral identity}

We make use of the following identity in the proof of
\autoref{p:YS}. It is an adaptation of
\cite[\S 4, Lemma 1.2]{DynPDE} to the context of
inhomogeneous killing (or branching) rates.

\begin{lemma}\label{l:dynkin}
  Let $k\from [0,\infty)^2 \to [0,\infty) $ be a function,
  left-differentiable in its first component,
  such that $k(r,t) = -k(t,r)$
  and that, for $0\le r\le t$, $\partial_r k(r,t)$ is independent of $t$.
  If
  \[
    w(r) = e^{-k(r,t)} u(t) + \int_r^t e^{-k(r,u)}v(u)\, \dd u,
  \]
  then
  \[
    w(r) = u(t) + \int_r^t v(s) \, \dd s - \int_r^t \partial_s k(r,s) w(s)\, \dd s.
  \]
\end{lemma}

\bibliographystyle{abbrvnat}

\end{document}